\documentclass[a4paper, reqno]{amsart}

\usepackage[all]{xy}
\usepackage[english]{babel}
\usepackage[utf8]{inputenc}
\usepackage{amssymb,amsmath,amsthm}
\usepackage{amsfonts}
\usepackage{graphicx}
\usepackage{psfrag}
\usepackage[dvipsnames]{xcolor}
\usepackage[hidelinks]{hyperref}
\usepackage{csquotes}
\usepackage[alphabetic]{amsrefs}
\usepackage{tikz-cd}
\usepackage{enumerate}

\allowdisplaybreaks

\newcommand{\End}{\mathrm{End}}

\newcommand{\Ric}{\mathrm{Ric}}

\newcommand{\scal}{\mathrm{scal}}
\newcommand{\e}{\epsilon}

\newcommand{\n}{\nabla}

\renewcommand{\L}{\mathcal{L}}
\newcommand{\E}{\mathcal E}

\renewcommand{\o}{\omega}

\newcommand{\tr}{\mathrm{tr}}
\renewcommand{\a}{\alpha}
\newcommand{\de}{\delta}

\renewcommand{\d}{\partial}

\newcommand{\abs}[1]{\left\lvert#1\right\rvert}
\newcommand{\norm}[1]{\left\lVert#1\right\rVert}
\newcommand{\ldr}[1]{\langle #1\rangle}

\renewcommand{\H}{\mathcal H}

\newcommand{\T}{\mathcal T}
\renewcommand{\E}{\mathcal E}
\renewcommand{\Re}{\mathrm{Re}}
\renewcommand{\O}{\mathcal O}
\newcommand{\on}{\mathring \n}
\newcommand{\mR}{\mathcal R}

\newenvironment{itquote}
  {\begin{quote}\itshape}
  {\end{quote}\ignorespacesafterend}

\theoremstyle{plain}
\newtheorem{thm}{Theorem}[section]
\newtheorem{prop}[thm]{Proposition}
\newtheorem{lemma}[thm]{Lemma}
\newtheorem{cor}[thm]{Corollary}
\theoremstyle{definition}
\newtheorem{definition}[thm]{Definition}

\newtheorem{remark}[thm]{Remark}
\newtheorem{example}[thm]{Example}

\newtheorem{assumption}[thm]{Assumption}

\newcommand{\R}[0]{\mathbb{R}}							
\newcommand{\C}[0]{\mathbb{C}}							
\newcommand{\K}[0]{\mathbb{K}}							
\newcommand{\N}[0]{\mathbb{N}}							
\newcommand{\Z}[0]{\mathbb{Z}}							
\newcommand{\D}[0]{\mathcal{D}}							

\newcommand{\U}[0]{\mathcal U}

\makeatletter
\@namedef{subjclassname@1991}{Subject}
\@namedef{subjclassname@2000}{Subject}
\@namedef{subjclassname@2010}{2020 Mathematics Subject Classification}
\makeatother

\newcommand{\dv}{\text{ }dV}
\newtheorem{lem}[thm]{Lemma}

\usepackage{geometry}

\title[Heat kernel estimates on ALE manifolds]{Long-time estimates for heat flows on asymptotically locally Euclidean manifolds}

\author{Klaus Kröncke}
\address{University of Hamburg, Department of Mathematics, Bundesstraße 55, 20146 Hamburg, Germany}
\email{klaus.kroencke@uni-hamburg.de}

\author{Oliver L. Petersen}
\address{Department of Mathematics, Uppsala University, Box 480, 75106 Uppsala, Sweden}
\email{oliver.petersen@math.uu.se}

\subjclass[2010]{Primary 35K08; Secondary 58J05; 53C25.}
\keywords{heat kernel, ALE manifolds, Schr\"{o}dinger operators}

\begin{document}
\hbadness=100000
\vbadness=100000

\begin{abstract}
We consider the heat equation associated to Schr\"{o}dinger operators acting on vector bundles on \emph{asymptotically locally Euclidean} (ALE) manifolds. 
Novel $L^p - L^q$ decay estimates are established, allowing the Schr\"{o}dinger operator to have a non-trivial $L^2$-kernel. 
We also prove new decay estimates for spatial derivatives of arbitrary order, in a general geometric setting.
Our main motivation is the application to stability of non-linear geometric equations, primarily Ricci flow, which will be presented in a companion paper.
The arguments in this paper use that many geometric Schr\"{o}dinger operators can be written as the square of Dirac type operators.
By a remarkable result of Wang, this is even true for the Lichnerowicz Laplacian, under the assumption of a parallel spinor.
Our analysis is based on a novel combination of the Fredholm theory for Dirac type operators on ALE manifolds and recent advances in the study of the heat kernel on non-compact manifolds.
\end{abstract}

\maketitle

\tableofcontents

\begin{sloppypar}
\section{Introduction}
Let $V$ be a vector bundle with metric connection $\nabla$ on a manifold $M^n$.
A Schr\"{o}dinger operator acting on sections of $V$
 is an operator of the form $\Delta_V = \nabla^*\nabla+\mathcal{R}$, where $\nabla^*$ is the formal adjoint of $\nabla$ and $\mathcal{R}$ is a smooth symmetric endomorphism field on $V$. Given $\Delta_V$, we have a natural evolution problem
\begin{align*}
\partial_tu + \Delta_V u=0,\qquad u|_{t=0}=u_0
\end{align*}
for the associated heat equation.
Assuming that the potential $\mR$ is bounded (or decays at infinity) and that $\Delta_V$ is non-negative in the $L^2$-sense, we are interested in the evolution operator
\begin{equation} \label{eq: heat operator}
	e^{-t\Delta_V}: L^p \to L^q
\end{equation}
mapping the initial data $u_0$ to the solution $u(t, \cdot)$, which we will, by abuse of notation, also call the heat kernel.
One would like to understand the following questions:
\begin{enumerate}[(i)]
	\item What are sharp decay estimates in time for the heat kernel \eqref{eq: heat operator}?
	\item Does each spatial derivative of the heat kernel decay faster?
	\item What if $\Delta_V$ has a non-trivial $L^2$-kernel (null space)?
	\item What if $V$ is a non-trivial vector bundle?
\end{enumerate}
In the case that $M = \R^n$ and $\Delta_V=\Delta$ is the scalar Laplacian or
$M$ is compact, the answers to all these questions are classical by now. 
There is a vast literature on these questions on general non-compact manifolds (see the surveys \cites{Cou03,Gri99,Sal10}), but
all work so far has only considered a subset of these questions, see \cites{CZ07,CDS16,Devyver2014,Devyver2018,S2013} for some important recent work.
The main purpose of this paper is to present new methods, which simultaneously treat all questions $\text{(i-iv)}$ on a large class of non-compact manifolds, known as \emph{asymptotically locally Euclidean} (ALE) manifolds.
In particular, we prove decay estimates for arbitrary order spatial derivatives, using the Fredholm theory for Dirac type operators on ALE manifolds.
This extends recent results (see e.g.\ \cites{CD03,Devyver2018,MO19}), where decay of the first order spatial derivative of the heat kernel on functions is proven. 

Our main motivation for this work is an application to Ricci flow.
The simplest non-trivial non-compact \emph{Ricci-flat} manifolds are ALE manifolds. 
In a companion paper, we are interested in proving their (non-linear) $L^p$-stability under Ricci flow.
The linearized Ricci flow equation can, after a choice of gauge, be put in the form \eqref{eq: heat operator}, where $\Delta_V$ is the Lichnerowicz Laplacian. 
Since the moduli space of Ricci-flat ALE manifolds is non-trivial in general, the Lichnerowicz Laplacian will typically have a non-trivial $L^2$-kernel, motivating $\text{(iii)}$.
In order to prove stability under Ricci flow, we need sharp estimates for the associated heat kernel, also for the derivatives, motivating $\text{(i)}$ and $\text{(ii)}$.
The Lichnerowicz Laplacian acts on symmetric $2$-tensors, which in general are sections in a non-trivial vector bundle, also at infinity, motivating $\text{(iv)}$.
The key observation we use is due to Wang in \cite{Wang91}, where he proves that the Lichnerowicz Laplacian can, provided the manifold carries a parallel spinor, be written as the square of a Dirac type operator. The assumption of having a parallel spinor leads to many other well-known geometric consequences which we will discuss and use in Section \ref{sec : geometric operators 2}.

Our main results in the present paper will, in particular, apply to any square of a Dirac type operator, with suitable asymptotic structure at infinity.
In addition to the Lichnerowicz Laplacian, the results apply to many geometric operators, including the Laplace-Beltrami, the Hodge-Laplace and classical and twisted Dirac operators.
In each case, one gets a slightly different decay result. 
This is closely related to growth and decay rates of harmonic sections at infinity.
We investigate this systematically throughout the paper.

\subsection{Geometric setup}
Before we explain the main results of this paper, let us introduce the geometric setting we are working in.
We define
\[
	\R^n_{> 1} := \R^n \backslash \overline{B_1} = (1, \infty) \times S^{n-1},
\]
where $B_1$ is the ball of radius $1$ around the origin.
We will from now on assume that $n \geq 3$.
Let $\Gamma$ be a finite subgroup of $SO(n)$, which acts freely on $S^{n-1}$.
The quotient $S^{n-1}/\Gamma$ is a smooth compact manifold with induced metric $g_{S^{n-1}/\Gamma}$.
We define the \emph{locally Euclidean cone}
\[
	\R^n_{> 1}/ \Gamma := (1, \infty) \times \left(S^{n-1}/\Gamma\right),
\]
with the flat metric $\mathring g = dr^2 + r^2g_{S^{n-1}/\Gamma}$ and Levi-Civita connection $\on$.
Given a tensor field $h$, the notation
\[
	h \in \O_\infty\left(r^\a\right)
\]
means that for each $k \in \N_0$, there is a constant $C_k > 0$, such that
\[
	\big|\on^kh\big|_{\mathring g} \leq C_k r^{\a - k}
\]
as $r \to \infty$ in $\R^n_{> 1}/ \Gamma$.
We consider manifolds which are \emph{asymptotic} to a locally Euclidean cone in the following sense:

\begin{definition}[ALE and AE manifolds] \label{def: ALE}
A complete Riemannian manifold $(M^n, g)$ is called \emph{asymptotically locally Euclidean} (ALE for short) with one end, if there is a compact subset $K \subset M$, a real number $\tau>0$ (called the order) and a diffeomorphism $\phi: M_\infty := M \backslash K \to \R^n_{> 1}/\Gamma$ such that
\begin{equation} \label{eq: decay metric}
	\phi_*g - \mathring g \in \O_\infty\left(r^{-\tau}\right).
\end{equation}
The diffeomorphism $\phi$ will also be called \enquote{coordinate system at infinity}. 
If $\Gamma=\left\{1\right\}$, we call $(M^n,g)$ \emph{asymptotically Euclidean} (AE for short).
\end{definition}
This definition naturally extends to ALE and AE manifolds with a finite number of ends. 
In this paper, we allow multiple ends unless stated otherwise. 
AE manifolds can be easily constructed by conformally blowing up a compact Riemannian manifold $(M,g)$ at a finite number of points. 
Similarly, ALE manifolds which can be constructed by conformally blowing up a compact Riemannian orbifold at its (finitely many) orbifold points. 
We are particularly interested in (but do not restrict our analysis to) \emph{Ricci-flat} manifolds, which are harder to construct, but do exist:

\begin{example}[Ricci-flat ALE manifolds]
The simplest example of a Ricci-flat ALE manifold (different from $\R^n$) is the Eguchi-Hanson manifold.
Let $\a_1, \a_2, \a_3$ be the standard left-invariant one-forms on $S^3$.
For each $\e > 0$, define the Eguchi-Hanson metric
\[
	g_{eh, \e} := \frac{r^2}{\left(r^4 + \e^4\right)^{\frac12}}\left(dr \otimes dr + r^2 \a_1 \otimes \a_1\right) + \left(r^4 + \e^4\right)^{\frac12}\left(\a_2 \otimes \a_2 + \a_3 \otimes \a_3\right)
\]
for $r >0$.
After we quotient by $\Z_2$, we can smoothly glue in an $S^2$ at $r = 0$ to get the (complete) Eguchi-Hanson manifold $(TS^2,g_{eh,\e})$, which is ALE with $\Gamma=\Z_2$ and hyperk\"{a}hler, hence Ricci-flat. 
This is an example of Kronheimer's classification of hyperk\"{a}hler ALE manifolds \cite{Kro89}: 
Each $4$-dimensional hyperk\"{a}hler ALE manifold is diffeomorphic to a minimal resolution of $(\R^4\setminus\left\{0\right\})/\Gamma$, where $\Gamma\subset \mathrm{SU}(2)$ be a discrete subgroup acting freely on $S^3$.
\end{example}

Let throughout the paper $V \to M$ be a real or complex vector bundle, equipped with a positive definite symmetric or Hermitian inner product $\ldr{\cdot,\cdot}$ and a compatible connection $\n$.
We will consider \emph{Schr\"{o}dinger operators} of the form
\[
	\n^*\n + \mR,
\]
where $\mR$ is a smooth symmetric endomorphism field of $V$.
We will assume that $\n^*\n + \mR$ is asymptotic to the flat Laplacian $-\sum_{i=1}^{n}\partial^2_{x_ix_i}$ at infinity in a suitable sense, so that we can apply the elliptic theory.
In order to do this, we need the restriction of $V$ to $M_\infty$ to be a quotient by $\Gamma$ of a trivial bundle $\R^n_{> 1} \times \K^m$, where $\K\in\left\{\R,\C\right\}$. 
In other words, we want $V$ to satisfy the same $\Gamma$-equivariance at infinity, as the manifold does:

\begin{assumption}[The vector bundle at infinity] \label{ass: V}
We assume that there is a representation
\begin{equation} \label{eq: representation} 
	\Gamma \to \End(\K^m),
\end{equation}
respecting the standard Riemannian or Hermitian inner product, and a bundle isomorphism $\Phi$, such that the following diagram commutes:
\[
	\begin{tikzcd}
V_\infty \arrow[r, "\Phi"] \arrow[d]
& \left(\R^n_{> 1} \times \K^m\right)/\Gamma \arrow[d] \\
M_\infty \arrow[r, "\phi"]
& \R^n_{> 1}/\Gamma
\end{tikzcd}.
\]
\end{assumption}

\begin{remark}
The representation \eqref{eq: representation} is in most geometric examples simply induced by the action of $\Gamma$ on $\R^n_{> 1}$, such examples include the tensor bundle and the spinor bundle.
Note, however, that the action of $\Gamma$ is indeed non-trivial in these examples.
\end{remark}

\subsection{Main results}
\subsubsection{Almost Euclidean heat kernel estimates}

We have the following assumptions on our Schr\"{o}dinger operator $\n^*\n + \mR$:

\begin{definition} \label{def: asympt Eucl Laplacian}
Let $(M,g)$ be an ALE manifold as in Definition \ref{def: ALE} and let $\Phi$ be a trivialization as in Assumption \ref{ass: V}.
A Schr\"{o}dinger operator
\[
	\Delta_V := \n^*\n + \mR,
\]
with a symmetric endomorphism field $\mR$, is said to be \emph{asymptotic to a Euclidean Laplacian} if 
\begin{align}
	\Phi_*\n - \on
		&\in \O_\infty\left(r^{- 1 - \tau}\right), \label{eq: nabla convergence} \\
	\Phi_* \mR
		&\in \O_\infty\left(r^{-2-\tau}\right), \label{eq: R convergence}
\end{align}
where $\on$ is the connection $\left(\R^n_{> 1} \times \K^m\right)/\Gamma$, induced by the trivial connection.
\end{definition}

\begin{remark}
For example the connection $\n$ could be the Levi-Civita connection on the tensor or spinor bundles and the potential $\mR$ could be expressed in terms of the curvature of $M$.
In these cases, \eqref{eq: nabla convergence} and \eqref{eq: R convergence} typically follow from \eqref{eq: decay metric}.
\end{remark}

\noindent
A Schr\"{o}dinger operator $\Delta_V$ which is asymptotic to a Euclidean Laplacian is self-adjoint on $L^2(M,V)$ and $\ker_{L^2}(\Delta_V)$ is finite dimensional (see Lemma \ref{le: projection onto kernel} for the latter statement).
For fixed $1 \leq p \leq q \leq \infty$, the natural heat kernel estimate one would like to have is 
\begin{equation} \label{eq: Euclidean estimate}
	\norm{e^{-t\Delta_V }u}_{L^q} \leq Ct^{-\frac n2 \left(\frac1p - \frac1q\right)}\norm{u}_{L^p}
\end{equation}
for $u\in L^p$ and $t>0$, where $C = C(p,q)$.
For the Euclidean Laplacian $\Delta$ on $\R^n$, this estimate holds for any $1 \leq p \leq q \leq \infty$ and any $u\in L^p$.
However, if the Schr\"{o}dinger operator has an $L^2$-kernel (null space), then \eqref{eq: Euclidean estimate} cannot hold true, since elements in the kernel are \emph{stationary} solutions to the heat equation.
However, since $\Delta_V$ is self-adjoint, $L^2$-orthogonality is preserved under the heat flow.
It is therefore a reasonable aim to prove decay for initial data \emph{$L^2$-orthogonal} to the kernel:

\begin{definition}
Let $\Delta_V=\nabla^*\nabla+\mathcal{R}$ be a Schr\"{o}dinger operator which is asymptotic to a Euclidean Laplacian, in the sense of Definition \ref{def: asympt Eucl Laplacian}.
\begin{itemize}
\item[(i)] Assume that $\ker_{L^2}(\Delta_V)=\left\{0\right\}$. The heat kernel is said to satisfy \emph{Euclidean heat kernel estimates} if for each ${1 \leq p \leq q \leq \infty}$ the estimate \eqref{eq: Euclidean estimate} holds, for each $u \in L^p$.
\item[(ii)] Assume that $\ker_{L^2}(\Delta_V) \subset \O_{\infty}\left(r^{-n}\right)$.
The heat kernel is said to satisfy \emph{almost Euclidean heat kernel estimates} if for each $1 < p \leq q < \infty$, the estimate \eqref{eq: Euclidean estimate} holds for each $u \in L^p$, which is $L^2$-orthogonal to $\ker_{L^2}(\Delta_V)$.
\end{itemize}
\end{definition} 
The condition $\ker_{L^2}(\Delta_V) \subset \O_{\infty}\left(r^{-n}\right)$ implies that $\ker_{L^2}(\Delta_V) \subset L^p$, for all $p \in (1,\infty]$. 
Thus, the orthogonal projection $\Pi:L^p\to \ker_{L^2}(\Delta_V)$ is well-defined for all $p\in [1,\infty)$.
Our first main result is the following:
\begin{thm}[Heat kernel estimate] \label{thm: main heat kernel estimate}
Let $(M,g)$ be an ALE manifold and $\Delta_V=\nabla^*\nabla+\mathcal{R}$ be a Schr\"{o}dinger operator on a vector bundle $V$ over $M$, which is asymptotic to a Euclidean Laplacian in the sense of Definition \ref{def: asympt Eucl Laplacian}.
Assume that
\[
	(\Delta_V u, u)_{L^2} \geq 0
\]
for all $u \in C_c^\infty$, and that $\ker_{L^2}(\Delta_V) \subset \O_{\infty}\left(r^{-n}\right)$.
Then $e^{-t\Delta_V}$ satisfies almost Euclidean heat kernel estimates.
\end{thm}
For a Schr\"{o}dinger operator which is asymptotic to a Euclidean Laplacian, one in general has $\ker_{L^2}(\Delta_V)\subset\O_{\infty}\left(r^{2-n}\right)$ (see e.g.\ \cite{pacini}).
The motivation for assuming that elements in the $L^2$-kernel of $\Delta_V$  decay as $r^{-n}$, comes from the fact that harmonic differential forms on ALE manifolds typically have this decay, see Proposition \ref{prop : decay harmonic forms}.
This may be applied in many geometric situations, which we present below.

\begin{remark}\label{rmk : Euclidean heat kernel estimates}
Our heat kernel estimate is nicely complemented by a result of Devyver \cite{Devyver2014}*{Thm.\ 3}, where he proves that Euclidean estimates hold if the $L^2$-kernel is trivial.
Note also that, in case $n \geq 9$, then Theorem \ref{thm: main heat kernel estimate} follows by a very general result of Devyver in \cite{Devyver2018}*{Thm.\ 1.7}.
After the first preprint version of this paper was completed, Devyver also generalized his results to lower dimensions, see the published version of \cite{Devyver2018}.
Let us also mention another interesting approach, using microlocal analysis, to obtain pointwise decay for the heat kernel (not derivatives) on functions by Sher in \cite{S2013}*{Thm.\ 2}, based on the resolvent estimates in \cite{GH08}. 
\end{remark}

\subsubsection{Derivative estimates}

Let us now turn to the derivative estimates for the heat kernel. 
Our method is based on the elliptic theory for Dirac type operators on ALE manifolds, c.f.\ Section \ref{sec: Elliptic Estimates}.
We will prove derivative estimates for Schr\"{o}dinger operator, which are squares of Dirac type operators of the following type:
\begin{definition} \label{def: asympt Eucl Dirac}
Let $\Phi$ be a trivialization as in Assumption \ref{ass: V}.
Consider a first order formally self-adjoint differential operator
\[
	\D_V := A\circ\n + B
\]
acting on sections of $V$, where $A$ is a smooth homomorphism field from $T^*M\otimes V$ to $V$
 and $B$ is a smooth endomorphism field on $V$.
Then $\D_V$ is called a \emph{Dirac type operator} if $\D_V^2$ is a Schr\"{o}dinger operator and called \emph{asymptotic to a Euclidean Dirac operator} if
\begin{align}
	\Phi_*A - A_0 &\in \O_\infty\left(r^{ - \tau}\right), \label{eq: decay A} \\
	\Phi_*B &\in \O_\infty\left(r^{ -1 - \tau}\right), \label{eq: decay B}
\end{align}
where $A_0$ is a constant homomorphism such that
\begin{equation} \label{eq: Delta at infinity}
	\on^*\on = A_0(\on) \circ A_0(\on)
\end{equation}
on $\left(\R^n_{> 1} \times \K^m\right)/ \Gamma$.
\end{definition}

\noindent
Important examples are the Hodge de Rham operator, the classical Dirac operator and twisted versions of it on ALE manifolds. 
Via an embedding of bundles, even the Lichnerowicz Laplacian is the square of the twisted Dirac operator on spinor valued one-forms, under the assumption of a parallel spinor.

\begin{remark} \label{rmk: Positivity Dirac squared}
Note that the assumptions \eqref{eq: decay A}, \eqref{eq: decay B} and \eqref{eq: Delta at infinity} imply that $\Delta_V := \D_V^2$ is a Schr\"{o}dinger operator, which is asymptotic to a Euclidean Laplacian, in the sense of Definition \ref{def: asympt Eucl Laplacian}.
The non-negativity in this case is automatic:
\[
	(\Delta_V u, u)_{L^2} = (\D_V u, \D_V u)_{L^2} \geq 0
\]
for all $u \in C^\infty_c$.
Moreover, if $\ker_{L^2}\left(\Delta_V\right) \subset \O_\infty\left(r^{-n}\right)$, then $\ker_{L^2}\left(\Delta_V\right) = \ker_{L^2}\left(\D_V\right)$.
\end{remark}

Motivated by the Euclidean case, one would ideally like to get derivative estimates
\begin{equation} \label{eq: derivative estimate}
\norm{\n^k e^{-t\Delta_V}u}_{L^p} \leq Ct^{-\frac k2}\norm{u}_{L^p}
\end{equation}
for $1 \leq p \leq \infty$, and all $u \in L^p$, which is $L^2$-orthogonal to $\ker_{L^2}(\Delta_V)$.
However, on non-flat manifolds the story becomes very delicate. 
We will show below that for Schr\"{o}dinger operators, which are squares of Dirac type operators as above, we in fact obtain \eqref{eq: derivative estimate} for a large range of $k$ and $p$.
Interestingly, there seems to be a certain threshold, beyond which we only get the following weaker estimate:
\begin{equation} \label{eq: weak derivative estimate}
	\norm{\nabla^k\circ e^{-t\Delta_V }u}_{L^{p}} \leq Ct^{-\frac{n}{2p}-\frac{l}{2}+\e} \norm{u}_{L^p}
\end{equation}
for all $t \geq t_0$, where $l\in\N_0$ is a fixed non-negative integer, the constants $\e>0$ and $t_0 > 0$ can be made arbitrarily small and $C=C(k,p,t_0,\epsilon)$.
In order to state our main result for the derivative of the heat kernel, let us make the following definition:

\begin{definition}\label{def : Derivative estimates}
We say that $e^{-t\Delta_V}$ satisfies \emph{derivative estimates of degree $l \in \N_0 \cup \{\infty\}$} if for each $k \in \N_0$, we have the estimate \eqref{eq: derivative estimate} for all
\begin{align}
	p &\in (1, \infty), \quad \text{if } k \leq l, \label{eq: derivative k leq l} \\
	p &\in \left(1, \frac n{k-l} \right), \quad \text{if } k \geq l+1, \label{eq: k geq l plus 1 p small}
\end{align} 
and the estimate \eqref{eq: weak derivative estimate} for all
\begin{align}
	p &\in (1, \infty) \backslash \left(1, \frac n{k-l} \right), \quad \text{if } k \geq l+1, \label{eq: k geq l plus 1 p large}
\end{align} 
and all $u \in L^p$, which are $L^2$-orthogonal to $\ker_{L^2}(\Delta_V)$. Note that the decay rates match nicely at $p=\frac{n}{k-l}$.
\end{definition}

As we will see later in Theorem \ref{thm : better derivative estimates introduction},
the degree $l$ corresponds, in a certain sense, to the growth rate of the \emph{slowest growing harmonic section}.
(In fact, we will be able to allow slower growing harmonic sections, but we then require that the covariant derivative of the section to an appropriate order vanishes.)
There might, in general, exist bounded harmonic sections on the ALE manifold.
Without further assumptions, one therefore would hope for derivative estimates of degree $l = 0$, which is our next main result:

\begin{thm}[Derivative estimates]\label{thm : Derivative estimates introduction}
Let $(M,g)$ be an ALE manifold and $\Delta_V=\D_V^2$, where $\D_V$ is a formally self-adjoint Dirac type operator on a vector bundle $V$ over $M$, which is asymptotic to a Euclidean Dirac operator in the sense of Definition \ref{def: asympt Eucl Dirac}.
Assume that $\ker_{L^2}(\Delta_V) \subset \O_{\infty}\left(r^{-n}\right)$.
Then $e^{-t\Delta_V}$ satisfies derivative estimates of degree $0$.
\end{thm}

To the best of our knowledge, this is the first result about the long-time behavior of arbitrarily high derivatives under the heat flow on ALE manifolds.
So far, derivatives of the heat flow have only been studied for Schr\"{o}dinger operators on functions and, in that case, only the first derivative, c.f.\ the discussion in Subsection \ref{sec: Riesz} below.
The following proposition provides a simple way of improving the degree $l$ in Theorem \ref{thm : Derivative estimates introduction}:
\begin{prop}\label{prop : better derivative estimates 0 introduction}
For all $0 \leq m \leq l$, assume that $\D_m$ are formally self adjoint Dirac type operators on $V\otimes T^*M^{\otimes m}$, which are asymptotic to a Euclidean Dirac operator, in the sense of Definition \ref{def: asympt Eucl Dirac}.
Assume also that
\begin{equation} \label{commuting}
\begin{split}
  \nabla\circ (\D_{m-1})^2=(\D_m)^2\circ\nabla, \qquad \nabla\circ\nabla^*\leq C\cdot (\D_m)^2,\qquad\nabla^*\circ\nabla\leq C \cdot(\D_{m-1})^2
\end{split}
\end{equation}
holds for $1\leq m \leq l$.
If
\begin{itemize}
	\item[(i)] $\ker_{L^2}(\D_l)\subset \O_{\infty}(r^{-n})$, then $e^{-t\D_0^2}$ satisfies derivative estimates of degree $l$,
	\item[(ii)] $\ker_{L^2}(\D_l) = \{0\}$, then $e^{-t\D_0^2}$ satisfies strong derivative estimates of degree $l$.
	Here, \enquote{strong} is explained in Definition \ref{def : strong derivative estimates} below.
\end{itemize}
\end{prop}

\begin{definition} \label{def : strong derivative estimates}
We say that $e^{-t\Delta_V}$ satisfies \emph{strong derivative estimates of degree $l \in \N_0$} if the estimates in Definition \ref{def : Derivative estimates} hold and, additionally, in case $k \geq l+1$ and $p \in (1, \infty) \backslash \left( 1, \frac n{k-l} \right]$, we may set $\e = 0$ in the estimate \eqref{eq: weak derivative estimate}.
\end{definition}

\begin{example}\label{rem : better derivative estimates 0}
The Hodge de Rham operator $d+d^*$ on $\R^n$ satisfies the assumptions of Proposition \ref{prop : better derivative estimates 0 introduction} for every $l\in\N$.
\end{example}

We give further applications of Theorem \ref{thm: main heat kernel estimate}, Theorem \ref{thm : Derivative estimates introduction} and Proposition \ref{prop : better derivative estimates 0 introduction} to various concrete geometric differential operators below.

\subsubsection{Improved derivative estimates}
Theorem \ref{thm : Derivative estimates introduction} can be improved by putting assumptions on the  harmonic sections. 
In order to explain the method, let us first recall that elements in $\ker\left(\Delta_V\right)=\ker\left(\D_V^2\right)$ have an expansion near infinity, in the integer growth rates
\[
	\hdots, r^{-n}, r^{1-n}, r^{2-n}, 1, r, r^2, \hdots,
\]
coming from the asymptotic rates of harmonic functions near infinity on $\R^n_{> 1}$. 
The idea of the improved derivative estimates can be loosely formulated as follows:
\begin{itquote}
If we can control the behavior of harmonic sections up to growth rate $l-1$, then derivative estimates hold of degree $l$.
\end{itquote}
More precisely, we would like to have is the following implication:
\begin{equation} \label{eq: weak implication derivatives}
	\forall k \leq l: \quad \D_V^k u = 0, \quad u = o \left(r^k\right) \quad \Longrightarrow \quad \n^k u = 0.
\end{equation}
Under this assumption, the main theorem below says that (weak) derivative estimates of degree $l$ hold.
Note that the stronger implication
\begin{equation} \label{eq: strong implication derivatives}
	\forall k \leq l: \quad \D_V^2 u = 0, \quad u = o \left(r^k\right) \quad \Longrightarrow \quad u = 0,
\end{equation}
clearly implies \eqref{eq: weak implication derivatives}.
\begin{remark}
Let us give an example, which shows that the implication \eqref{eq: strong implication derivatives} really is less general than \eqref{eq: weak implication derivatives}:
We choose $M=\R^n$ and $\D_V = d + d^*$ on differential forms. 
Any linear function $u$ on $\R^n$ solves $\D_V^2 u = \Delta u = 0$. 
Therefore \eqref{eq: weak implication derivatives} does not hold for $l = 1$.
However, $\n^2 u = 0$, so \eqref{eq: strong implication derivatives} is true for $l = 1$.
In fact, it is even true for any $l\in\N$: One can show that any function $u$ satisfying $(d+d^*)^ku=0$ and $u\in o(r^k)$ is a polynomial of degree $\leq k-1$, c.f.\ the computations in \cite{BGM71}. Therefore, $\nabla^ku=0$ as well.
\end{remark}

In many applications, we are only interested in the decay for the heat kernel on \emph{a part} of the vector bundle $V$.
For example, we think of forms of a certain degree or as the spinors of positive or negative chirality.
We will include this in the main theorem of this section. Before we state the result, we need to introduce a slightly weaker version of Definition \ref{def : Derivative estimates}:
\begin{definition}\label{def : weak derivative estimates}
We say that $e^{-t\Delta_V}$ satisfies \emph{weak derivative estimates of degree $l \in \N$} if the estimates in Definition \ref{def : Derivative estimates} are satisfied for all $p$ and $k$, except if
\begin{align}
\begin{split}
	p
		&\in \left\{\frac{n}{k},\frac{n}{k-1},\ldots,n\right\}, \quad \text{if }k\leq l, \\
	p
		&\in \left\{\frac{n}{k},\frac{n}{k-1},\ldots,\frac{n}{k-l+1}\right\}, \quad \text{if } k \geq l+1.
\end{split} \label{eq: k leq l exceptional}
\end{align}
in which case we instead have the following  slightly weaker estimate: 
For each $\epsilon>0$ and each $t_0>0$, there exists a constant ${C = C(k,p,\epsilon,t_0})$, such that
\[
	\norm{\nabla^k \circ e^{-t\Delta_V }u}_{L^p} \leq Ct^{-\frac k2+\epsilon}\norm{u}_{L^p}
\]
holds for every $u$, which is $L^2$-orthogonal to $\ker_{L^2}(\Delta_V)$ and for every $t \geq t_0$.
\end{definition}
In other words, \enquote{weak} means that we loose an arbitrarily small $\e > 0$ in decay at certain exceptional values of $p$. Now we can state the main result of this section.
\begin{thm}[Improved derivative estimates]\label{thm : better derivative estimates introduction}
Let $(M,g)$ be an ALE manifold and $\D_V$ be a formally self-adjoint Dirac type operator on a vector bundle $V$ over $M$, which satisfies the assumptions of Theorem \ref{thm : Derivative estimates introduction}. Let $E \subset V$ be a parallel subbundle such
 that
\begin{align}\label{eq: D^2 preserves E}
	\D_V^2 u \in C^\infty(E)
\end{align}
 for all $u\in C^{\infty}(E)$. Assume additionally that there exists an $l \in \N$ such that the implication
 \begin{equation}
	\forall k \leq l: \quad \D_V^k u = 0, \quad u = o \left(r^k\right) \quad \Longrightarrow \quad \n^k u = 0
\end{equation}
 holds for all $u\in C^{\infty}(E)$.
Then, $e^{-t(\D_V)^2}|_E$ satisfies weak derivative estimates of degree $l$.
\end{thm}

As already mentioned, it is useful to remember that \eqref{eq: strong implication derivatives} implies \eqref{eq: weak implication derivatives}.
In other words, if there are no harmonic sections of growth less than $l$, then weak derivative estimates hold of degree $l$.

For the Laplace-Beltrami operator on functions, the following simple special case of Theorem \ref{thm : better derivative estimates introduction} follows:

\begin{cor}[Improved derivative estimates for the heat kernel on functions]\label{cor : better derivative estimates}
Let $(M,g)$ be an ALE manifold and $\Delta$ be its Laplace Beltrami operator acting on functions.
Assume that there is an $l \in \N_0$, such that if $u \in o(r^l)$ and $\Delta u$ is constant, then $u$ is constant (and thus $\Delta u = 0$).
Then $e^{-t\Delta}$ satisfies weak derivative estimates of degree $l$.
\end{cor}

In \cite{CCH06}, the authors studied the Riesz transform on AE manifolds and show that it is bounded on $L^p$ if  $p<n$. As a consequence of this result (c.f.\ also the discussion in Subsection \ref{sec: Riesz} below), 
\[
\norm{\nabla \circ e^{-t\Delta}}_{p \to p}\leq C t^{-\frac{1}{2}}
\]
for $p\in (1,n)$. 
On the other hand, if the manifold has more than one end, they show unboundedness of the Riesz transform for $p\geq n$ which is due to the existence of non-constant bounded harmonic functions on the manifold.
Thus, they are not able to establish this estimate for $p\geq n$.

Our Theorem \ref{thm : better derivative estimates introduction} states that a connection between optimality of derivative estimates and non-existence of non-constant harmonic functions exists in a much greater generality: We found that the optimality of estimates on  the k'th derivative relies on the non-existence of k-harmonic sections with growth rate less than $r^k$, whose k'th derivative vanishes.
\begin{example}
The condition \eqref{eq: weak implication derivatives} is satisfied for the trivial line bundle $E=\Lambda^0\R^n \cong \R$ as a subbundle of $V=\Lambda \R^n$ with the Hodge de Rham operator $d+d^*$.
\end{example}

\subsubsection{A discussion about the threshold in the derivative estimates} \label{sec: Riesz}
A long-standing problem in harmonic analysis is the question for which $p$ the \textit{Riesz transform}
\begin{align}\label{eq : Riesz transform}
\nabla\circ \Delta^{-\frac{1}{2}}:L^p(M)\to L^p(T^*M)
\end{align}
is a bounded operator. 
This problem has been studied in the context of derivative estimates, (see e.g. the papers \cites{Devyver2018,CDS16}), but also as an independent problem on asymptotically Euclidean \cite{CCH06} and asymptotically conical manifolds \cites{GH08,GH09}.
In fact, it is straightforward to see that \eqref{eq : Riesz transform} implies the optimal derivative estimate
\begin{align}\label{eq : optimal_derivative_estimates}
	\norm{\nabla\circ e^{-t\Delta}}_{p \to p}
		\leq \norm{\nabla\circ \Delta^{-\frac{1}{2}}}_{p \to p} \norm{\Delta^{\frac{1}{2}}\circ e^{-t\Delta}}_{p \to p}
		\leq C t^{-\frac{1}{2}}.
\end{align}
However, the converse implication \eqref{eq : optimal_derivative_estimates} $\Rightarrow$ \eqref{eq :  Riesz transform} was established in \cite{ACDH}*{Thm.\ 1.3}, provided that the volume doubling property and the scaled Poincar\'{e} inequalities do hold. Both conditions are satisfied for ALE manifolds with one end.

The latter result suggests a strong link between derivative estimates of arbitrary order and elliptic estimates. 
As the proofs of our results build on optimal elliptic estimates and invertibility of elliptic operators, we believe that the decay rates we established are optimal (possibly up to removing the arbitrarily small factor $\e>0$).

\begin{remark}\label{rem : interpretation decay rates}
A heuristic way of interpreting the heat kernel estimates and derivative estimates of degree $l$ is the following:
For simplicity, assume that $u$ is rotationally symmetric at infinity. 
Then $u\in L^p$ means that $u=o(r^{-\frac{n}{p}})$, i.e.\ $L^p$ corresponds to the spatial decay rate $r^{-\frac{n}{p}}$. Similarly, $L^q$ corresponds to the spatial decay rate $r^{-\frac{n}{q}}$. On the other hand the norm of the map 
	\[
	e^{-t\Delta_V}:L^p\to L^q
	\] 
has the temporal decay rate	$t^{-\frac n2\left(\frac1p-\frac1q\right)}=t^{\frac12\left(-\frac np-\left(-\frac nq\right)\right)}$. 
Heuristically, one gets this temporal decay rate by taking the difference of the two spatial decay rates $-\frac np$ and $-\frac nq$, corresponding to $L^p$ and $L^q$, and multiplying by a factor $1/2$, which comes from parabolic rescaling $(t,r)\mapsto (\alpha^2t,\alpha r)$. 
Moreover, if $\nabla^ku\in L^p$ (i.e. $\nabla^ku=o(r^{-\frac{n}{p}})$),
one heuristically has $u=o(r^{k-\frac{n}{p}})$ due to elliptic estimates for weighted Sobolev spaces.
If $k\in\N$ is so small that $k-\frac{n}{p}<l$, then the norm  of the map
\[
	\nabla^k\circ e^{-t\Delta_V}:L^p\to L^p
\] 
has temporal decay rate $t^{-\frac{k}{2}}=t^{\frac{1}{2}\left(-\frac np-\left(k-\frac np\right)\right)}$. In this case, one gets this temporal rate by taking the difference of the spatial decay (and potentially growth) rates $o(r^{-\frac{n}{p}})$ and $o(r^{k-\frac{n}{p}})$ and again multiplying by the factor $1/2$. 
If $k\in\N$ is so large that $k-\frac{n}{p}\geq l$, the norm  of the map
	\[
	\nabla^k\circ e^{-t\Delta_V}:L^p\to L^p
	\] 
has (up to an arbitrary small $\epsilon>0)$ temporal decay rate $t^{-\frac{n}{2}\frac{1}{p}-\frac{l}{2}}=t^{\frac12\left(-\frac{n}{p}-l\right)}$. This can be heuristically explained as follows: There is a critical growth rate $o(r^{l})$, $l\in\N_0$ (which is, in view of Theorem \ref{thm : Derivative estimates introduction}, the growth rate of a harmonic section whose behavior we can not control) which acts as a barrier. In this case, the temporal decay rate is obtained by taking the difference of the spatial decay rate $r^{-\frac{n}{p}}$ and the critical spatial growth rate $r^l$ (which is smaller than $r^{k-\frac{n}{p}}$) and again multiplying by the factor $1/2$.
\end{remark}

\subsubsection{Application to geometric operators}\label{subsubsec : geometric operators}\label{subsec : applications}

The first prominent example to which we apply Theorem \ref{thm: main heat kernel estimate} and Theorem \ref{thm : Derivative estimates introduction} is the Hodge Laplacian:
\begin{cor}\label{cor : heat flow hodge laplace}
Let $(M,g)$ be an ALE manifold and $\Delta_H=(d+d^*)^2$ be the Hodge Laplacian on the exterior algebra $\Lambda M$. Suppose that $\mathcal{H}_1(M):=\ker_{L^2}(\Delta_{H}|_{T^*M})=\left\{0\right\}$. Then, $e^{-t\Delta_H}$, satisfies almost Euclidean heat kernel estimates and derivative estimates of degree $0$.
\end{cor}
This result requires in addition a careful analysis of the decay of harmonic forms at infinity. 
Using this result and various bundle identifications for special holonomy metrics, we can study further operators under an additional geometric condition.
\begin{cor}\label{cor : heat flow ricci flat}
Let $(M,g)$ be an ALE spin manifold which carries a parallel spinor. 
Then, almost Euclidean heat kernel estimates and derivative estimates of degree $0$ are satisfied by the following two operators:
\begin{itemize}
\item[(i)] $e^{-t(\D_{T^*M})^2}$, where $\D_{T^*M}$ is the twisted Dirac operator on spinor-valued one-forms.
\item[(ii)] $e^{-t\Delta_L}$, where $\Delta_L$ is the Lichnerowicz Laplacian on symmetric $2$-tensors.
\end{itemize}
\end{cor}
In addition, we also prove estimates for the linearized Ricci curvature $h\mapsto \frac{d}{dt}\Ric_{g+th}$ and the linearized de Turck vector field $h\mapsto \frac{d}{dt}V(g+th,g)$ (where $V(g,h)=\mathrm{tr}_g({}^{g}\nabla-{}^{h}\nabla)$) along $e^{-t\Delta_L}$. These will be relevant for applications in Ricci flow.

Using Proposition \ref{prop : better derivative estimates 0 introduction} and Theorem \ref{thm : better derivative estimates introduction}, we obtain better estimates for restrictions of the above operators to subbundles:
\begin{cor}\label{cor : heat flow functions}
	Let $(M,g)$ be an ALE manifold  with $\mathcal{H}_1(M)=\left\{0\right\}$ and $\Delta$ be its Laplace-Beltrami operator. Then $e^{-t\Delta}$ satisfies Euclidean heat kernel estimates and strong derivative estimates of degree $1$. Moreoever, if $(M,g)$ is not AE, then $e^{-t\Delta}$ satisfies derivative estimates of degree $2$.
\end{cor}
\begin{cor}\label{cor : heat flow one forms}
	Let $(M,g)$ be an ALE manifold  and $\Delta_{H_1}$ be the Hodge Laplacian on $T^*M$. 
	\begin{itemize}
	\item[(i)] If $\mathcal{H}_1(M)=\left\{0\right\}$, $e^{-t\Delta_{H_1}}$ satisfies Euclidean heat kernel estimates and strong derivative estimates of degree $0$.
	\item[(ii)] If $\Ric\geq0$, $e^{-t\Delta_{H_1}}$ satisfies weak derivative estimates of degree $1$.
	\item[(iii)] If $(M,g)$ carries a parallel spinor, $e^{-t\Delta_{H_1}}$ satisfies derivative estimates of degree $1$.
	\end{itemize}
\end{cor}
Note that the assumptions in (iii) imply the assumptions in (ii) which in turn imply the assumptions in (i).

\begin{cor}\label{cor : heat flow dirac operator}
Let $(M,g)$ be an ALE spin manifold with non-negative scalar curvature and let $\D$ be its Dirac operator acting on spinors. Then, $e^{-t\D^2}$ satisfies Euclidean heat kernel estimates and weak derivative estimates of degree $1$.
Moreover, if $(M,g)$ carries a parallel spinor, $e^{-t\Delta_{H_1}}$ satisfies derivative estimates of degree $1$.
\end{cor}

\subsection{Structure of the paper}

The paper is structured as follows: In Section \ref{sec: Elliptic Estimates}, we discuss elliptic estimates for Dirac type operators on ALE manifolds.
In particular, we state scale-broken estimates for Dirac type operators, which are analogous to the corresponding results of \cite{Bartnik1986} for the Laplacian. 
These play a pivotal role for developing derivative estimates under heat flows.

Section \ref{sec: long time estimates} is the technical core of the paper and splits into three subsections. In Subsection \ref{subsec: heat kernel estimates}, we prove Theorem \ref{thm: main heat kernel estimate}.
In Subsection \ref{subsec: commuting operators}, we prove a general decay result for first order differential operators, which admit good commutation properties, composed with heat flows. 
Subsections \ref{subsec: derivative estimates} and \ref{subsec: improved derivative estimates}
we combine the results of the previous subsections with elliptic estimates and isomorphism properties of Dirac type operators on weighted Sobolev spaces to prove Theorem \ref{thm : Derivative estimates introduction}, Proposition \ref{prop : better derivative estimates 0 introduction}, Theorem \ref{thm : better derivative estimates introduction} and Corollary \ref{cor : better derivative estimates} and.

The main result of Section \ref{sec : decay harmonic forms} is Proposition \ref{prop : decay harmonic forms}, which proves an improved decay for harmonic $k$-forms on ALE manifolds, especially for $k\neq 1,n-1$. The result follows from elliptic theory, combined with a careful analysis of closed and coclosed forms on Euclidean space, carried out in Subsection \ref{subsec : decay harmonic forms}. 
In the final two Sections \ref{sec : geometric operators 1} and \ref{sec : geometric operators 2}, we are going to apply all these results to natural geometric diffential operators and prove the results in Subsection \ref{subsubsec : geometric operators}.

\vspace{3mm}
\noindent\textbf{Acknowledgements.}
The authors wish to thank Baptiste Devyver for useful discussions.
Part of this work was carried out while the authors were visiting the Institut Mittag-Leffler during the program \emph{General Relativity, Geometry and Analysis} in Fall 2019, supported by the Swedish Research Council under grant no.\ 2016-06596.
We wish to thank the institute for their hospitality and for the excellent working conditions provided. 
The work of the authors is supported through the DFG grant KR 4978/1-1 in the framework of the priority program 2026: \textit{Geometry at infinity}.

\section{Elliptic estimates for Dirac type operators} \label{sec: Elliptic Estimates}

In this section, we collect elliptic estimates in weighted Sobolev norms for Dirac type operators $\D_V$ which are asymptotic to a Euclidean Dirac operator, in the sense of Definition \ref{def: asympt Eucl Dirac}. The key steps are carried out modulo obvious modifications as in \cite{Bartnik1986} for Laplace type operators on AE manifolds and so we omit the details here.

Fix a point $x \in M$ and define the weight function 
\[
	\rho(y) := \sqrt{1 + d(y, x)^2},
\]
where $d$ is the Riemannian distance.
For any ${k \in \N_0}$, ${p \in [1, \infty)}$ and any ${\delta \in \R}$, the weighted Sobolev space $W_\delta^{k,p}(M)$ is the space of $V$-valued sections $u \in W^{k,p}_{loc}(M)$ such that
\[
	\norm{u}_{k, p, \delta} := \sum_{j = 0}^k \left(\int_{M}\abs{(\rho\n)^ju}^p\rho^{-\delta p - n}\dv\right)^{1/p}
\]
is finite. 
We also use the notation $L_\delta^p := W_\delta^{0,p}$. 
For $k\leq l$ and $\delta_1\leq \delta_2$, it is immediate that 
\[
\norm{u}_{k,p,\delta_2}\leq \norm{u}_{l,p,\delta_1}
\]
and for $\delta_1<\delta_2$ and $q\in (p,\infty)$, an application of the H\"{o}lder inequality yields
\begin{align*}
\norm{u}_{k,p,\delta_2}\leq \norm{1}_{0,r,\delta_2-\delta_1}\norm{u}_{k,q,\delta_1}\leq C\norm{u}_{k,q,\delta_1},
\end{align*}
where $r\in (p,\infty)$ was chosen so that $\frac{1}{p} = \frac{1}{q} + \frac{1}{r}$.
We will use both inequalities frequently throughout the paper.
The following standard Sobolev estimate is an analogue of \cite{Bartnik1986}*{Prop.\ 1.6}) which holds for any weight $\de$:
\begin{prop} \label{prop: dirac non Fredholm estimate}
	Let $k,m\in \N$, such that $m\leq k$, $\delta \in \R$ and $p \in (1, \infty)$.
	There is a constants $C > 0$ (depending on $n, \delta$ and $p$), such that
	\[
	\norm{u}_{k, p, \de} 
	\leq C( \norm{\D_V^m u}_{k - m, p, \de-m} +  	\norm{u}_{k-m, p, \de})
	\]
	for all $u \in W^{k, p}_\de(M)$.
\end{prop}
As an easy consequence, we get
\begin{cor} \label{cor: elliptic dirac gradient estimate 3}
	Let $k\in\N_0$ and $p\in (1,\infty)$.
	Then for any $q\in [p,\infty)$ satisfying $\frac{1}{p}-\frac{1}{q}\in \left[0,\frac{k}{n}\right)$, we have
	\[
	\norm{\n^k u}_{L^p(M)} \leq C\left(\norm{\D_V^k u}_{L^p(M)} + \norm{u}_{L^q(M)}\right)
	\]
	for all $u \in W^{k, p}_\de(M)$, where $\de := k-\frac np$. 
\end{cor}
\begin{proof}	Due to the assumptions, we have $q \geq p$ and $\delta \geq -\frac{n}{q}$. Due to Proposition \ref{prop: dirac non Fredholm estimate} and the H\"{o}lder inequality, we then get
	\begin{align*}
		\norm{\nabla^k u}_{L^p(M)}
		\leq \norm{u}_{k,p,\delta}
		&\leq C\left(\norm{\D_V^k u}_{{0,p,\delta-k}} + \norm{u}_{{0,p,\delta}}\right) \\
		&\leq C\left(\norm{\D_V^k u}_{{0,p,\delta-k}} +\norm{u}_{{0,q,-\frac{n}{q}}}\right) \leq C\left(\norm{\D_V^k u}_{L^p(M)} +\norm{u}_{L^q(M)}\right),
	\end{align*}
	which finishes the proof.
\end{proof}
For our purposes, we need a slight refinement of Corollary \ref{cor: elliptic dirac gradient estimate 3}, based on scale-broken elliptic estimates which do only hold for non-exceptional weights.
\begin{definition}
	We define the sets 
	\begin{align*}
		\E_1 
		:= \{k \in \Z \mid k \neq -1, -2, \hdots, 2-n\}, \qquad\qquad
		\E_2 
		:= \E_1 \cup \{2-n\}.
	\end{align*}
\end{definition}
Recall that $\E_2$ is the set of exceptional weights for the Laplace operator, see \cite{Bartnik1986}. This relies on the fact that
\[
	\on^*\on = A_0(\on) \circ A_0(\on): W_\delta'^{k+2,p} \to W_{\delta-2}'^{k,p}
\]
is an isomorphism for  $k \in \N_0$, $p \in (1, \infty)$ and that $\delta \in \R \backslash \E_2$, see  \cite{Bartnik1986}*{Thm.\ 1.7}. Recall that the spaces $W_\delta'^{k,p}$ are weighted Sobolev spaces of functions on $\R^n$ which use the radius function $r(x)=|x|$ instead of $\rho$ as a weight.
Using either $(\on^*\on)^{-1} \circ A_0(\on)$ or $A_0(\on) \circ (\on^*\on)^{-1}$ as an inverse for $A_0(\on)$, one quickly sees that
\[
	A_0(\on) : W_\delta'^{k+1,p} \to W_{\delta-2}'^{k,p}
\]
is an isomorphism for  $k \in \N_0$, $p \in (1, \infty)$ and that $\delta \in \R \backslash \E_1$.
Using this fact, one deduces as in \cite{Bartnik1986}*{Thm.\ 1.10} the following scale-broken Sobolev estimate:
\begin{prop} \label{prop: dirac Fredholm estimate}
Let $k,m\in \N$, such that $m\leq k$ and $\delta, \hdots, \de-(m-1) \in \R \backslash \E_1$ and let $p \in (1, \infty)$.
There are constants $C > 0$ and $R > 0$ (depending on $n, \delta$ and $p$), such that
\[
	\norm{u}_{k, p, \de} 
		\leq C \norm{\D_V^m u}_{k - m, p, \de-m} + C \norm{u}_{L^p(B_R)}
\]
for all $u \in W^{k, p}_\de(M)$.
\end{prop}

We get the following important gradient estimate:
\begin{cor} \label{cor: elliptic dirac gradient estimate}

 Let $k\in\N$, $p\in (1,\infty)$ and assume that $k-\frac{n}{p}\notin\N_0$.
Then
\[
	\norm{\n^k u}_{L^p(M)} \leq C\left(\norm{\D_V^k u}_{L^p(M)} + \norm{u}_{L^p(B_R)}\right)
\]
for all $u \in W^{k, p}_\de(M)$, with $\de := k-\frac np$.
\end{cor}
\begin{proof}
For $\de := k-\frac np$, we have
\[
	\norm{\n^k u}_{L^p(M)} \leq \norm{u}_{k,p,\de},
	\qquad 	\norm{u}_{L^p(M)} = \norm{u}_{0,p,\de-k},
\]
for any section $u$.
Due to the assumption, note that
\[
	\delta = k - \frac np, \hdots, \de - (k-1) = 1 - \frac np \in \R \backslash \E_1,
\]
since $1 - \frac np > 1-n$.
Thus, Proposition \ref{prop: dirac Fredholm estimate} implies the statement.
\end{proof}
A direct consequence, which turns out to be useful, is:
\begin{cor} \label{cor: elliptic dirac gradient estimate 2}
Let $k \in \N_0$, $p \in (1,\infty)$ and assume that $k - \frac np \notin \N_0$.
Then for any $q \in [p,\infty]$, we have
\[
	\norm{\n^k u}_{L^p(M)} \leq C\left(\norm{\D_V^k u}_{L^p(M)} + \norm{u}_{L^q(B_R)}\right)
\]
for all $u \in W^{k, p}_\de(M)$, with $\de := k-\frac np$.
\end{cor}

\section{Long-time estimates for heat flows of Schr\"{o}dinger operators}\label{sec: long time estimates}

In this section, we prove our main general results on the heat kernel and its derivatives: Theorem \ref{thm: main heat kernel estimate}, Theorem \ref{thm : Derivative estimates introduction}, Proposition \ref{prop : better derivative estimates 0 introduction} and Theorem \ref{thm : better derivative estimates introduction}.

\subsection{Heat kernel estimates}\label{subsec: heat kernel estimates}
We work now under the assumptions of Theorem \ref{thm: main heat kernel estimate}, which is the result we prove in this subsection.
In particular, we have assumed that
\[
	\ker\left(\Delta_V\right) \subset \O_\infty\left(r^{-n}\right) \subset L^p(M)
\]
for each $p \in (1, \infty)$.
The key observation for the $L^2$-kernel, which is crucial for our analysis, is the following:
\begin{lemma} \label{le: projection onto kernel}
The $L^2$-projection 
\[
	\Pi : L^p(M) \to \ker_{L^2}(\Delta_V) \subset L^q(M)
\]
is a finite range operator, in particular it is bounded, for all $p \in [1, \infty)$ and $q \in (1, \infty]$.
\end{lemma}

\begin{proof}
That the $L^2$-kernel is finite dimensional follows from standard Fredholm theory, for example, by a slight generalization of \cite{Bartnik1986}*{Prop.\ 2.2}.
Let $e_1, \hdots, e_m$ be an $L^2$-orthonormal basis of $\ker_{L^2}(\Delta_V)$.
Then we get the estimate
\begin{align*}
	\norm{\Pi(u)}_{L^q} 
		\leq \sum_{i = 1}^m \abs{\ldr{u, e_i}_{L^2}}\norm{e_i}_{L^q} 
		\leq \sum_{i = 1}^m \norm{u}_{L^p} \norm{e_i}_{L^{p'}} \norm{e_i}_{L^q},
\end{align*}
which is bounded since $p' = \frac p{p-1} > 1$.
\end{proof}
We will work with the operator
\[
	\Delta_V + \a \Pi,
\]
which is self-adjoint on $L^2$ and has trivial kernel, for each $\a > 0$.
The main heat kernel estimate for our analysis is the following:

\begin{thm}[Heat kernel estimate] \label{thm: heat kernel estimate}
There is an $\a_0 > 0$, such that for each $\a \leq \a_0$ and for each $1 < p \leq q < \infty$, we have
\[
	\norm{e^{-t(\Delta_V + \a \Pi)}}_{p \to q} \leq Ct^{-\frac n2 \left(\frac1p - \frac1q\right)}
\]
for some $C = C(p, q)$.
\end{thm}

Our main decay result on the heat kernel would follow as a simple corollary:

\begin{proof}[Proof of Theorem \ref{thm: main heat kernel estimate}, assuming Theorem \ref{thm: heat kernel estimate}]
We only need to note that
\[
	e^{-t(\Delta_V + \a \Pi)} u = e^{-t\Delta_V} u
\]
for all $u \in L^p$, which are $L^2$-orthogonal to $\ker_{L^2}(\Delta_V)$, i.e.\ those $u$ with $\Pi (u) = 0$.
\end{proof}
\begin{remark}
	As we pointed out in the introduction, Theorem \ref{thm: heat kernel estimate} follows from \cite{Devyver2018}*{Thm.\ 1.7} under the assumption $n\geq 9$. 
	Our proof follows a very similar approach and so we will not present details everywhere. 
	We will later mention why the assumption  $n\geq 9$ can be dropped.
\end{remark}
\subsubsection{The resolvent estimate}

The main ingredient needed in order to prove Theorem \ref{thm: heat kernel estimate} is the following resolvent estimate:

\begin{prop}[Resolvent estimate] \label{prop: resolvent}
For each $1 < p \leq q < \infty$, we have
\[
	\norm{(\Delta_V + \a \Pi + \lambda)^{-1}}_{p \to q} \leq C\lambda^{-\frac n2\left(\frac1p - \frac1q\right) - 1}
\]
for all $\lambda > 0$.
\end{prop}

As in \cite{Devyver2014}, let us write
\[
	\Delta_V = \n^*\n + \mR_+ - \mR_-,
\]
where $\mR_+$ and $\mR_-$ are non-negative endomorphisms of $V$.
We define
\[
	\H := \n^*\n + \mR_+, 
\]
for any $\lambda > 0$ and write
\begin{align*}
	\Delta_V + \a \Pi + \lambda
		= \H + \lambda + \a \Pi - \mR_- 
		= (\H + \lambda) (1 - \underbrace{(\H+\lambda)^{-1}(\mR_- - \a \Pi)}_{=: \T_\lambda}),
\end{align*}
which implies that
\begin{equation} \label{eq: resolvent identity}
	\left(\Delta_V + \a \Pi + \lambda\right)^{-1} = (1-\T_\lambda)^{-1}(\H + \lambda)^{-1}.
\end{equation}
The resolvent estimate, Proposition \ref{prop: resolvent}, can be divided into two steps, resolvent estimates for $\H$ and proving that $(1-\T_\lambda)^{-1}$ is bounded.
The point is that $\H$ is a positive perturbation of $\n^*\n$, and resolvent estimates for such operators are classical.
We have the following estimate in our situation:

\begin{lemma}[The resolvent estimate for $\H$] \label{le: resolvent H}
For all $1 \leq p \leq q \leq \infty$, we have
\[
	\norm{(\H + \lambda)^{-1}}_{p \to q}
		\leq C\lambda^{\frac n2\left(\frac1p - \frac1q\right) - 1}.
\]
\end{lemma}
\begin{proof}
This is immediate from the following computation, where we use \cite{Devyver2014}*{Cor.\ 1}:
\begin{align*}
	\norm{(\H + \lambda)^{-1}}_{p \to q}
		&= \norm{\int_0^\infty e^{-t(\H + \lambda)}dt}_{p \to q} 
			\leq \int_0^\infty \norm{e^{-t(\H + \lambda)}}_{p \to q} dt 
			\leq \int_0^\infty e^{-t\lambda}\norm{e^{-t\H}}_{p \to q} dt \\
		&\leq C \int_0^\infty e^{-t\lambda}t^{-\frac n2\left(\frac1p - \frac1q\right)}dt 
			\leq C \lambda^{\frac n2\left(\frac1p - \frac1q\right)-1} \int_0^\infty e^{-s}s^{\frac n2\left(\frac1p - \frac1q\right)}dt
			\leq C \lambda^{\frac n2\left(\frac1p - \frac1q\right)-1}. 
\end{align*}
\end{proof}

Given the resolvent estimate for $\H$, it suffices to prove the boundedness in $L^p$ for the operator $(1 - \T_\lambda)^{-1}$:

\begin{lemma} \label{le: compensating resolvent}
The
\[
	{(1 - \T_\lambda)^{-1}}: L^p \to L^p
\]
is bounded for all $p \in (1, \infty)$, independently of $\lambda$.
\end{lemma}
\begin{proof}[Sketch of proof]
	This result corresponds to \cite{Devyver2018}*{Prop.\ 2.4}, where it only holds for dimensions $n\geq9$ in a more general setting. In our situation, we can use the strong assumption $\ker_{L^2}(\Delta_V) \subset \O_{\infty}\left(r^{-n}\right)$ which implies that $\ker_{L^2}(\Delta_V) \subset L^q$ for any $q>1$, as stated earlier. In fact this allows us to use the approach in \cite{Devyver2014}*{Sec.\ 3} which carries over, with only small modifications, to our setting. We assume that the operator of the Lemma can be written as a Neumann series
	\begin{align}\label{eq:neumann}
	(1 - \T_\lambda)^{-1} = \sum_{m = 0}^\infty \T_\lambda^m
	\end{align}
	and we will prove convergence of the right hand side in every $L^p$-space. Recall that
	\[
		\T_\lambda = \H_\lambda^{-1}(\mR_- + \a \Pi ).
	\]
	Now one first shows that for each $p, q \in (1, \infty)$, there is an $N = N(p, q) \in \N$, such that
	\begin{align}\label{eq:resolvent_mapping}
		\T_\lambda^N: L^p \to L^q
	\end{align}
	is bounded, with the bound independent of $\lambda \geq 0$.
This is analogous to \cite{Devyver2014}*{Prop.\ 6}, which is the same assertion for $p,q\in [\frac{n}{n-2},\infty]$. The proof carries over, up to a slight adaption of \cite{Devyver2014}*{Lem.\ 5}, taking the different mapping properties of $\mR_- + \a \Pi$ into account. Now as in \cite{Devyver2014}*{Lem.\ 6}, one proves that there is a constant $\epsilon>0$ such that
\begin{align}\label{eq:resolvent_powers}
\norm{ \T_\lambda^k}_{L^{\frac{2n}{n-2}},L^{\frac{2n}{n-2}}}\leq C(1-\epsilon)^k
\end{align}
for every $k\in\N$. This enables us to consider large powers of $T_\lambda$ as compositions of operators
\begin{equation*}
	\T_\lambda^m: L^p {\xrightarrow{\T_\lambda^{N_1}}}L^{\frac{2n}{n-2}} {\xrightarrow{\T_\lambda^{m-N_1-N_2}}} L^{\frac{2n}{n-2}}
	{\xrightarrow{\T_\lambda^{N_2}}} L^p.
\end{equation*}
By combining \eqref{eq:resolvent_mapping} and \eqref{eq:resolvent_powers}, one easily gets convergence of the series on the right hand side in \eqref{eq:neumann}, which proves the lemma.
\end{proof}

\begin{proof}[Proof of Proposition \ref{prop: resolvent}]
The proof follows by combining Lemma \ref{le: resolvent H} and Lemma \ref{le: compensating resolvent}.
\end{proof}

In order to use the resolvent estimate to prove Theorem \ref{thm: heat kernel estimate}, we need a basic $L^2$-estimate for the heat operator:

\begin{lemma} \label{le: heat kernel 2 to 2}
We have
\[
	\norm{\left(1 + t\left(\Delta_V + \a \Pi \right)\right) e^{-t\left(\Delta_V + \a \Pi \right)}}_{2 \to 2} \leq C.
\]
\end{lemma}
\begin{proof}
This is a consequence of the spectral theorem which can also be established by showing that 
\begin{align*}
	\frac{d}{dt} \left( t\ldr{\left(\Delta_V + \a \Pi\right) u, u}_{L^2}+ \ldr{ u, u}_{L^2}\right)\leq 0
	\end{align*}
for $u := e^{-t\left(\Delta_V + \a \Pi\right)} u_0$, where $u_0\in L^2$.
The details are left to the reader.
\end{proof}

Using this, we may now apply the resolvent estimate, Proposition \ref{prop: resolvent}, to obtain the following lemma:

\begin{lemma} \label{lemma: the two heat kernel estimates}
For all $p \in (1, \infty)$, we have
\begin{align*}
	\norm{e^{-t\left(\Delta_V + \a \Pi \right)}}_{p \to p}
		\leq C.
\end{align*}
Moreover, for all $1 < p \leq 2 \leq q < \infty$, we have
\begin{align*}
	\norm{e^{-t\left(\Delta_V + \a \Pi \right)}}_{p \to q}
		\leq C t^{-\frac n2 \left(\frac1p - \frac1q\right)}.
\end{align*}
\end{lemma}
\begin{proof}
By Proposition \ref{prop: resolvent}, with $t := \frac1\lambda$, and Lemma \ref{le: heat kernel 2 to 2}, we have the estimate
\begin{align*}
	\norm{e^{-t\left(\Delta_V + \a \Pi \right)}}_{p \to 2} 
		&\leq \norm{e^{-t\left(\Delta_V + \a \Pi \right)}\left(1 + t\left(\Delta_V + \a \Pi\right)\right)}_{2 \to 2} \norm{\left(1 + t\left(\Delta_V + \a \Pi\right)\right)^{-1}}_{p \to 2} \\
		&\leq \norm{\left(1 + t\left(\Delta_V + \a \Pi\right)\right)e^{-t\left(\Delta_V + \a \Pi \right)}}_{2 \to 2} \norm{\left(1 + t\left(\Delta_V + \a \Pi\right)\right)^{-1}}_{p \to 2} \\
		&\leq C t^{-\frac n2 \left(\frac1p - \frac12\right)}
\end{align*}
for all $p \in (1, 2]$. In particular, the first assertion holds for $p=2$.
\cite{Devyver2018}*{Prop.\ 2.6} implies the first assertion for $p\in (1,2)$ and duality implies the same assertion for $p\in (2,\infty)$.
By duality, we also conclude that 
\begin{align*}
	\norm{e^{-t\left(\Delta_V + \a \Pi \right)}}_{2 \to q} 
		&\leq C t^{-\frac n2 \left(\frac12 - \frac1q\right)},
\end{align*}
for all $q \in [2, \infty)$. 
We get
\begin{align*}
	\norm{e^{-t\left(\Delta_V + \a \Pi \right)}}_{p \to q}
		&\leq \norm{e^{-\frac t2\left(\Delta_V + \a \Pi \right)}}_{2 \to q} \norm{e^{-\frac t2\left(\Delta_V + \a \Pi \right)}}_{p \to 2} 
		\\		&
		\leq C t^{-\frac n2 \left(\frac12 - \frac1q\right)}t^{-\frac n2 \left(\frac1p - \frac12\right)}
		 = C t^{-\frac n2 \left(\frac1p - \frac1q\right)}
\end{align*}
as claimed.
\end{proof}
\begin{proof}[Proof of Theorem \ref{thm: heat kernel estimate}]
	The proof follows from combining the results we obtained with a standard interpolation argument which we state for completeness.
Let us first assume that $1 < p \leq q \leq 2$ and let $u := e^{-t(\Delta_V + \a \Pi)}u_0$. Choose $\theta \in (0, 1)$ such that
\[
	\frac1q = \frac{1-\theta}p + \frac\theta2.
\]
Then, interpolation and Lemma \ref{lemma: the two heat kernel estimates} gives
\begin{align*}
	\norm{u}_{L^q}
		&\leq \norm{u}_{L^p}^{1-\theta}\norm{u}_{L^2}^\theta
			\leq C \norm{u}_{L^p}^{1-\theta} t^{-\frac n2 \theta \left(\frac1p - \frac12\right)} \norm{u}_{L^2}^\theta
			\leq C t^{-\frac n2 \left(\frac1p - \frac1q\right)} \norm{u_0}_{L^p}.
\end{align*}
Duality implies the case $2 \leq p \leq q < \infty$ and concatenating these estimates gives the remaining case $1 < p \leq 2 \leq q < \infty$, which finishes the proof of Theorem \ref{thm: heat kernel estimate}.
\end{proof}

\subsection{Commuting operators}\label{subsec: commuting operators}

In this section, we lay the groundwork for our derivative estimates.
We assume that we have a second vector bundle $W \to M$, with the same assumptions as for $V$.
The following theorem is the main result here:
\begin{thm}\label{thm_special_derivative_estimates}
Consider two Schr\"{o}dinger operators 
\begin{align*}
	\Delta_V := \n^*\n + \mR, \qquad \qquad
	\Delta_W := \overline \n^*\overline \n + \overline \mR,
\end{align*}
on $V$ and $W$, respectively, which are both assumed to satisfy the assumptions of Theorem \ref{thm: main heat kernel estimate}.
Let $P:C^{\infty}(M,V)\to C^{\infty}(M,W)$ be a first-order differential operator such that
	\[ 
		P \circ \Delta_V
			= \Delta_W \circ P
	\]
	and such that there exist two constants $C_1,C_2>0$ satisfying
	\begin{align}\label{eq: two constants}
		P^* \circ P 
			\leq C_1 \cdot \Delta_V,
		\qquad\qquad P \circ P^* 
			\leq C_2 \cdot \Delta_W. 
	\end{align}
	Then for all $1<p\leq q<\infty$, and all $k \ \in \N_0$, there are constants $C = C(n, k, p, q)$ such that
\begin{align*}
	\norm{P e^{-t\Delta_V}}_{p\to q}
		\leq C t^{-\frac{n}{2}(\frac{1}{p}-\frac{1}{q}) - \frac12}, \qquad\qquad
	\norm{P^*e^{-t\Delta_W}}_{p\to q}
		\leq C t^{-\frac{n}{2}(\frac{1}{p}-\frac{1}{q})-\frac12}.
\end{align*}
\end{thm}

\begin{remark}
The theorem applies directly to the case when $\Delta_V = P^*P$ and $\Delta_W = PP^*$, in particular when $\Delta_V = (\D_V)^2$.
However, we need this general formulation in later applications.
\end{remark}

\begin{remark} \label{rmk: commutator W}
Note that both the assumptions as well as the assertion remain the same if we simultaneously interchange $P$ and $P^*$ as well as $\Delta_V$ and $\Delta_W$.
\end{remark}

In the special case where $\Delta_V=\Delta_W$ are both the Hodge Laplacian on $\Lambda M$ and $P$ is the Hodge de Rham operator, the assertion of Theorem \ref{thm_special_derivative_estimates} already follows from 
\cite{Devyver2018}*{Thm.\ 1.17} and \cite{GS15}*{Cor.\ 5}. 
It turned out that the proof in \cite{Devyver2018}*{Thm.\ 1.17} as well as the proofs of results in \cites{D92,CS08} on which it builds upon can be translated with small modifications to this much more general situation. 
We therefore omit details of the proofs in some intermediate results and give precise references instead.
The present form of the theorem is however crucial for proving our derivative estimates.
We start with the observation that an $L^2$-bound implies a bound on the principal symbol.
\begin{lemma}\label{lem : principal symbol}
The principal symbol $\sigma(P)\in C^{\infty}(M,\mathrm{Hom}(T^*M\otimes V,W))$ of $P$ satisfies the inequality
\begin{align*}
|\sigma(P)_x|\leq \sqrt{C_1}\qquad \text{for all }x\in M,
\end{align*}
where $C_1 > 0$ is the constant in estimate \eqref{eq: two constants} and where
\begin{align*}
|\sigma(P)_x|:=\sup\left\{ |\sigma(P)_x(\xi)v| \mid\xi\in T_xM,v\in V_x, |\xi|=|v|=1\right\}.
\end{align*}
\end{lemma}

\begin{proof}
Without loss of generality, let us assume that $V$ and $W$ are complex vector bundles.
Assume, in order to reach a contradiction, that there is a unit length $(\xi, v)$ such that
\begin{equation} \label{eq: contradiction assumption}
	|\sigma(P)(\xi) v| > \sqrt{C_1}.
\end{equation}
Choose coordinates $(x_1, \hdots, x_n)$ on an open subset $\U \ni x$, such that
\[
	dx_1|_x = \xi.
\]
For a first order operator, we have
\[
	Pu = \sigma(P)(\n u) + B(u)
\]
for some smooth endomorphism field $B$.
For any $u \in C^\infty_c(\U, V)$, we compute
\begin{align*}
	\ldr{P^*P\left(e^{inx_1}u\right), e^{i n x_1}u}_{L^2}
		&= \norm{P\left(e^{i n x_1} u\right)}_{L^2}^2 \\
		&\geq \norm{i n \sigma(P)(d x_1) u + Pu}_{L^2}^2 \\
		&\geq n^2 (1 - \de) \norm{\sigma (P)(d x_1) u}_{L^2}^2 + \left( 1 - \frac1\de \right)\norm{Pu}_{L^2}^2,
\end{align*}
and similarly
\begin{align*}
	\ldr{\Delta_V \left(e^{i n x_1} u\right), e^{i n x_1}u}_{L^2}
		&\leq \norm{\n\left(e^{i n x_1}u\right)}_{L^2}^2 + \norm{\mathcal R}_{L^\infty(\U)} \norm{u}_{L^2}^2 \\
		&\leq n^2 (1 + \de) \norm{\abs{dx_1}u}_{L^2}^2 + \left(1 + \frac1\de\right) \norm{\n u}_{L^2}^2 + C \norm{u}_{L^2}^2
\end{align*}
for any $\de \in (0, 1)$.
Applying \eqref{eq: two constants}, dividing by $n^2$ and letting $n \to \infty$, we conclude that
\[
	(1 - \de) \norm{\sigma (P)(d x_1) u}_{L^2}^2 \leq C_1 (1 + \de) \norm{\abs{dx_1}u}_{L^2}^2
\]
for any $\de \in (0, 1)$.
Letting $\de \to 0$, we get
\[
	 \norm{\sigma (P)(d x_1) u}_{L^2} \leq \sqrt{C_1} \norm{\abs{dx_1}u}_{L^2}.
\]
Assuming that $u|_x = v$ and by shrinking the support of $u$ near $x$, we reach a contradiction to \eqref{eq: contradiction assumption}, which completes the proof.
\end{proof}

The next step is to prove an $L^2$-estimate with a weight function $\phi = e^{\a \psi}$, which we will later choose carefully.

\begin{lemma}\label{lemma_weighted_L2}
Let $\psi:M\to\R$ be a smooth bounded function with $|d\psi|\leq1$, $\alpha\in\R$ and $\phi=e^{\alpha\psi}:M\to\R$. Then, there exists a constant $C>0$ such that
	\begin{align*}
		\norm{\phi P e^{-t\Delta_V}u}_{L^2}
			&\leq C t^{-\frac12} e^{C\a^2t}\left\|\phi u\right\|_{L^2}^2
	\end{align*}
for all $u\in L^2(M,V)$.
\end{lemma}

\begin{remark}
Choosing $\a = 0$ in the above lemma implies that $\phi = 1$, which proves the theorem for $p = q = 2$.
\end{remark}

\begin{proof}
This proof generalizes \cite{D92}*{Lem.\ 1} which is a result for the scalar heat equation.
Note first that for any $v \in H^k$, $k\geq2$, integration by parts yields
\begin{align*}
	(\phi \nabla^*\nabla v,\phi v)_{L^2}
		&=(\nabla v,\nabla(\phi\cdot \phi v))_{L^2}
			=(\nabla v,\phi\nabla(\phi v))_{L^2}+(\nabla v,[\nabla,\phi]\phi v)_{L^2}\\
		&=(\nabla(\phi v),\nabla(\phi v))_{L^2}-([\nabla,\phi]v,\nabla(\phi v))_{L^2}\\
		&\qquad +(\nabla (\phi v),[\nabla,\phi] v)_{L^2}-([\nabla,\phi] v,[\nabla,\phi] v)_{L^2} \\
		&=(\nabla^*\nabla(\phi v),\phi v)_{L^2} -\alpha^2\norm{\abs{d\psi} \phi v}^2_{L^2},
\end{align*}
from which it follows that
\begin{align}
	(\phi \Delta_V v,\phi v)_{L^2}
		&=(\Delta_V(\phi v),\phi v)_{L^2} -\alpha^2\norm{\abs{d\psi} \phi v}^2_{L^2}. \label{eq: commutator with test function} 
\end{align}
This identity will be used several times throughout the proof.
For a $u_0 \in L^2$, let ${u := e^{-t\Delta_V}u_0}$. 
By Lemma \ref{le: heat kernel 2 to 2}, we know that $u \in H^k$, for all $k\in\N_0$ and $t>0$, so we can apply \eqref{eq: commutator with test function} to note that
\begin{align}
	\frac{d}{dt} \norm{\phi u}^2_{L^2}
		&= 2\Re(\phi \d_t u, \phi u)_{L^2} 
			= - 2\Re(\phi \Delta_V u, \phi u)_{L^2} \nonumber \\
		&= - 2\Re(\Delta_V \phi u, \phi u)_{L^2} + 2\alpha^2\norm{\abs{d\psi} \phi u}^2_{L^2} 
			\leq - \frac2{C_1}\norm{P\phi u}^2_{L^2} + 2\alpha^2\norm{\phi u}^2_{L^2}. \label{eq: alpha estimate 1}
\end{align}
We want to combine this with
\begin{align}
		\frac1{C_1} \norm{\phi P u}^2_{L^2} 
			&\leq \frac2{C_2} \norm{[P,\phi] u}^2_{L^2} + \frac2{C_2} \norm{P \phi u}^2_{L^2}. \label{eq: alpha estimate 2}
\end{align}
By definition of the principal symbol, we have
\begin{align*}
	[P,\phi]u
		=\sigma_P(\nabla\phi)u=\a \phi \sigma_P(\nabla \psi)u.
\end{align*}
Due to $|\nabla\psi|\leq 1$ and Lemma \ref{lem : principal symbol}, we can estimate 
\begin{align}
	\norm{[P,\phi]u}_{L^2}
		\leq \a \norm{\sigma_P}_{L^\infty}\norm{\phi u}_{L^2} \leq C \a \norm{\phi u}_{L^2}. \label{eq: alpha estimate 3}
\end{align}
Combining the estimates (\ref{eq: alpha estimate 1} - \ref{eq: alpha estimate 3}), we get
\begin{align}
	\frac{d}{dt} \norm{\phi u}^2_{L^2} \leq - \frac1{C_1}\norm{\phi P u}^2_{L^2} + C\alpha^2\norm{\phi u}^2_{L^2}. \label{eq: alpha estimate full 1}
\end{align}
Using that
\[
	 \Delta_W 
		\geq \frac1{C_2} P P^* \geq 0
\]
and equation \eqref{eq: commutator with test function}, note that
\begin{align}
	\frac{d}{dt} \norm{\phi Pu}^2_{L^2}
		&= 2\Re(\phi P\d_t u, \phi P u)_{L^2} 
			= - 2\Re(\phi P\Delta_Vu, \phi P u)_{L^2} \nonumber\\
		&= - 2\Re(\phi \Delta_WPu, \phi P u)_{L^2}
			= - 2\Re(\Delta_W\phi Pu, \phi P u)_{L^2} + 2\alpha^2\norm{\abs{d\psi} \phi Pu}^2_{L^2} \nonumber\\
		&\leq 2\alpha^2\norm{\phi Pu}^2_{L^2}. \label{eq: alpha estimate full 2}
\end{align}
Defining the energy
\[
	\E(u) :=\left\|\phi u\right\|_{L^2}^2+\frac{t}{C_1}\left\|\phi Pu\right\|_{L^2}^2,
\]
and combining the estimates \eqref{eq: alpha estimate full 1} and \eqref{eq: alpha estimate full 2}, we get
\begin{align*}
	\frac{d}{dt}\E(u)
		&\leq C\alpha^2\E(u),
\end{align*}
which completes the proof.
\end{proof}
We observe the following consequence:
\begin{lemma}\label{lemma_special_derivatives}
	The theorem holds for $p\in(1,2]$ and $q\in [2,\infty)$.
\end{lemma}
\begin{proof}
Note that  due to \eqref{eq: two constants}, $\ker_{L^2}(\Delta_V)\subset \ker_{L^2}(P)$ and $ \ker_{L^2}(\Delta_W)\subset \ker_{L^2}(P^*)$. Therefore,
\[
P\circ e^{-t\Delta_V}=P\circ e^{-t(\Delta_V+\alpha\Pi_V)},\qquad P^*\circ e^{-t\Delta_W}=P^*\circ e^{-t(\Delta_W+\alpha\Pi_W)}.
\]
By duality, this implies
\[
 e^{-t\Delta_V}\circ P^*= e^{-t(\Delta_V+\alpha\Pi_V)}\circ P^*,\qquad  e^{-t\Delta_W}\circ P=e^{-t(\Delta_W+\alpha\Pi_W)}\circ P,
\]
from which we conclude that
\begin{align*}
	P\circ e^{-t\Delta_V}
		&=e^{-\frac{t}{3}(\Delta_W+\alpha\Pi_W)}\circ P\circ e^{-\frac{t}{3}\Delta_V}\circ  e^{-\frac{t}{3}(\Delta_V+\alpha\Pi_V)}.
\end{align*}
and by applying Theorem \ref{thm: main heat kernel estimate} and Lemma \ref{lemma_weighted_L2} with $\a = 0$,
we obtain for  $p\in (1,2]$ and $q\in [2,\infty)$
that
\begin{align*}
	\left\|P\circ e^{-t\Delta_V} \right\|_{p \to q}
		&\leq
\left\|e^{-\frac{t}{3}(\Delta_W+\alpha\Pi_W)}\right\|_{L^2,L^q}
\left\|P\circ e^{-\frac{t}{3}\Delta_V}\right\|_{L^2,L^2}\left\| e^{-\frac{t}{3}(\Delta_V+\alpha\Pi_V)}
\right\|_{L^p,L^2}\\
		&\leq C t^{-\frac{n}{2}(\frac{1}{2}-\frac{1}{q})}\cdot t^{-\frac{1}{2}}\cdot t^{-\frac{n}{2}(\frac{1}{p}-\frac{1}{2})} 
			=C t^{-\frac{n}{2}(\frac{1}{p}-\frac{1}{q})-\frac{1}{2}},
\end{align*}
as claimed. The proof of other estimate is completely analogous.
\end{proof}

Another consequence of Lemma \ref{lemma_weighted_L2} is the following:

\begin{lemma}[The localized $L^2 - L^2$ estimate]\label{lemma_local_L2}
Let $A,B\subset M$ be measurable subsets of $M$ and  $\chi_A,\chi_B$ be the characteristic function of $A,B$, respectively. 
Let $v\in C^{\infty}(V,M)$ and suppose that $\chi_A v\in L^2(V,M)$. 
Then there is a constant $C > 0$, such that
	\begin{align*}
		\left\|\chi_B P e^{-t\Delta_V}\chi_A\cdot v\right\|_{L^2}
			&\leq Ct^{-1/2}e^{-\frac{d(A,B)^2}{Ct}}\left\|\chi_A\cdot v\right\|_{L^2}.
	\end{align*}
\end{lemma}
\begin{proof}
	A similar proof has been done for the scalar heat equation in  \cite{D92}*{Thm.\ 2}.
	Choose a bounded smooth function $\psi$ such that $\psi|_A=0$, $\psi|_B=\frac{d(A,B)}{2}$ and $|\nabla\psi|\leq1$ and let $\phi=e^{\alpha \psi}$.
	Lemma \ref{lemma_weighted_L2} implies that for any $f\in L^2(W)$ we have
	\begin{align*}
		\langle \chi_B P e^{-t\Delta_V}\chi_Av, f\rangle_{L^2}
			&= \langle \phi P e^{-t\Delta_V}\chi_A\cdot v,\phi^{-1} \chi_B f\rangle_{L^2} 
				\leq \left\| \phi P e^{-t\Delta_V}\chi_A\cdot v \right\|_{L^2} \left\| \phi^{-1}\chi_B f \right\|_{L^2}\\
			&\leq Ct^{-\frac12}e^{C\a^2t}\left\|\phi\chi_A\cdot v\right\|_{L^2} e^{-\frac{\alpha}{2} d(A,B)} \left\| f \right\|_{L^2} \\	
			&\leq Ct^{-\frac12}e^{C\a^2t - \frac{\alpha}{2} d(A,B)}\left\|\phi\chi_A\cdot v\right\|_{L^2} \left\| f \right\|_{L^2}.
	\end{align*}
	If we set $\alpha=\beta\frac{d(A,B)}t$ and $\beta>0$ such that $C\beta^2-\frac{\beta}{2}<0$, we obtain 
	\begin{align*}
		\langle \chi_B v, f\rangle_{L^2}\leq C t^{-1/2}e^{-\frac{d(A,B)^2}{Ct}}\left\|\chi_A\cdot v\right\|_{L^2}\left\| f \right\|_{L^2}.
	\end{align*}
	Because $f \in L^2(W)$ was arbitrary, this finishes the proof.
\end{proof}

\begin{lemma}[The localized $L^2 - L^p$ estimate]\label{lemma_local_Lp}
Let $A,B,\chi_A,\chi_B$ and $v$ be as in Lemma \ref{lemma_local_L2}.
Then for $p\in [2,\infty)$, there is a constant $C >0$, such that
\begin{align*}
\left\|\chi_B P e^{-t\Delta_V}\chi_A v\right\|_{L^p}
	\leq Ct^{-\frac{n}{2}(\frac{1}{2}-\frac{1}{p})-\frac{1}{2}}e^{-\frac{d(A,B)^2}{Ct}}\left\|\chi_A v\right\|_{L^2}.
\end{align*}
Moreover, for all $p \in (1, 2]$, we have
\begin{align*}
\left\|\chi_B P e^{-t\Delta_V}\chi_A v\right\|_{L^2}
	\leq Ct^{-\frac{n}{2}\left((\frac{1}{p} - \frac{1}{2}\right)-\frac{1}{2}}e^{-\frac{d(A,B)^2}{Ct}}\left\|\chi_A v\right\|_{L^p}.
\end{align*}
\end{lemma}
\begin{proof}
In \cite{Devyver2018}*{p.\ 43-45}, the same result was shown for $\Delta_V=\Delta_W$ being the Hodge Laplacian on $\Lambda M$ and $P=d+d^*$. It carries over without problems to this setting.
\end{proof}

\begin{lemma}\label{lemma_boundedLp}
Theorem \ref{thm_special_derivative_estimates} holds with
\[
	p = q \in (1, \infty).
\]
\end{lemma}
\begin{proof}
	The lemma follows from 
	partitioning $M$ into countably many subsets whose diameter grows in time with rate $\sqrt{t}$ and applying Lemma \ref{lemma_local_Lp} to these subsets.
	The details are as in the proof of \cite{CS08}*{Cor.\ 4.16}.
 \end{proof}
\begin{proof}[Proof of Theorem \ref{thm_special_derivative_estimates}]
	As in the proof of Lemma \ref{lemma_special_derivatives}, we write
\begin{align*}
P\circ e^{-t\Delta_V}&=P\circ e^{-\frac{t}{2}\Delta_V}\circ e^{-\frac{t}{2}\Delta_V}
=e^{-\frac{t}{2}\Delta_W}\circ P\circ e^{-\frac{t}{2}\Delta_V}=e^{-\frac{t}{2}(\Delta_W+\alpha\Pi_W)}\circ P\circ e^{-\frac{t}{2}\Delta_V}
\end{align*}
and similarly,
\[
P^*\circ e^{-t\Delta_W}=e^{-\frac{t}{2}(\Delta_V+\alpha\Pi_V)}\circ P^*\circ e^{-\frac{t}{2}\Delta_W}.
\]
Let now $1<p\leq q<\infty$. By the assumptions of Theorem \ref{thm_special_derivative_estimates} and by Lemma \ref{lemma_boundedLp}, we have
\[
\left\|P\circ e^{-t\Delta_V} \right\|_{p\to q}\leq
\left\|e^{-\frac{t}{2}(\Delta_W+\alpha\Pi_W)}\right\|_{p\to q}
\left\|P\circ e^{-\frac{t}{2}\Delta_V}\right\|_{p\to p}
\leq C t^{-\frac{n}{2}(\frac{1}{p}-\frac{1}{q})}\cdot t^{-\frac{1}{2}}=C t^{-\frac{n}{2}(\frac{1}{p}-\frac{1}{q})-\frac{1}{2}}.
\]
The estimate
\[
\left\|P^*\circ e^{-t\Delta_W} \right\|_{p\to q}\leq  C t^{-\frac{n}{2}(\frac{1}{p}-\frac{1}{q})-\frac{1}{2}}
\]
is shown completely analogously.	
\end{proof}

\subsection{Derivative estimates}\label{subsec: derivative estimates}
Our goal here is to prove our first main derivative estimate, Theorem \ref{thm : Derivative estimates introduction}.
This is where we begin to combine the Fredholm theory for the Dirac type operators and Theorem \ref{thm_special_derivative_estimates}.
We work with Schr\"{o}dinger operators, which are squares of Dirac type operators
\[
	\Delta_V = (\D_V)^2,
\]
where $\D_V$ is asymptotic to a Euclidean Dirac operator, in the sense of Definition \ref{def: asympt Eucl Dirac}.
The following lemma is merely a corollary of Theorem \ref{thm_special_derivative_estimates}:
\begin{lem}\label{lem : higher derivatives}
For each $p\in (1,\infty)$ and $k\in \N$, we have
\[
\norm{\D_V^k e^{-t\D_V^2}}_{p \to p}\leq C t^{-\frac k2}
\]
for some $C=C(k,p)$.
\end{lem}
\begin{proof}
Because $\D_V$ commutes with $e^{-t\D_V^2}$, we have
\[
	\D_V^k\circ e^{-t\D_V^2}
		= \D_V\circ e^{-\frac{t}{k}\D_V^2}\circ\ldots\circ \D_V\circ e^{-\frac tk\D_V^2}
		= \left(\D_V\circ e^{-\frac{t}{k}\D_V^2}\right)^k.
\]
Applying Theorem \ref{thm_special_derivative_estimates} to $\Delta_V =\Delta_W =\D_V^2$ and $P=\D_V$ thus implies
\[
	\norm{\D_V^k \circ e^{-t\D_V^2}}_{p \to p}
		\leq \norm{\D_V \circ e^{-\frac tk\D_V^2}}_{p \to p}^k
		\leq Ct^{-\frac k2},
\]
which finishes the proof of the lemma.
\end{proof}
We finally are ready to prove the first main derivative estimate:
\begin{proof}[Proof of Theorem \ref{thm : Derivative estimates introduction}]
For a given $u_0 \in L^p$, which is $L^2$-orthogonal to $\ker_{L^2}\left(\D_V^2\right)$, let us write $u=e^{-t\D_V^2 }u_0$.
By Remark \ref{rmk: Positivity Dirac squared}, we may apply Theorem \ref{thm: heat kernel estimate}, which implies that almost Euclidean heat kernel estimates hold, i.e.\ the case $k = 0$ is thus proven.
We turn to the case $k \geq 1$. 
Let first $p \in \left(1, \frac nk \right)$, which is only a non-empty set if $k \leq n-1$.
Note that $k-\frac{n}{p}\in (1-n,0)$, so the assumptions in Corollary \ref{cor: elliptic dirac gradient estimate 2} are satisfied.
Corollary \ref{cor: elliptic dirac gradient estimate 2}, Lemma \ref{lem : higher derivatives} and Theorem \ref{thm: main heat kernel estimate} yield
\[
	\norm{\n^k u}_{L^p}
		\leq C\left(\norm{\D_V^k u}_{L^p} + \norm{u}_{L^q}\right)
		\leq Ct^{-\frac{k}{2}}\norm{u_0}_{L^p} + Ct^{-\frac n2\left(\frac{1}{p}-\frac{1}{q}\right)}\norm{u_0}_{L^p}
\]
for any $q \in [p, \infty)$. 
Now, since $p < \frac nk$, there is a $q \in [p, \infty)$, large enough so that
\[
	\frac n2 \left(\frac1p - \frac1q \right) = \frac k2.
\]
Inserting this $q$ in the above estimate, we conclude that
\[
	\norm{\n^k u}_{L^p}
		= Ct^{-\frac{k}{2}}\norm{u_0}_{L^p},
\]
as claimed.
We now let $p \in (1, \infty) \backslash \left(1, \frac nk \right)$ and fix an $\e > 0$. 
Choose $q\in (p,\infty)$ so large that $\frac{n}{2q}<\epsilon$.
Since $\frac{1}{p}-\frac{1}{q}\in [0,\frac{k}{n})$, Corollary \ref{cor: elliptic dirac gradient estimate 3} and Lemma \ref{lem : higher derivatives} yield
\[
	\norm{\n^k u}_{L^p} 
		\leq C\left(\norm{\D_V^k u}_{L^p} + \norm{u}_{L^q}\right)
		\leq C\left(t^{-\frac k2}+t^{-\frac n2\left(\frac1p-\frac1q\right)}\right)\norm{u_0}_{L^p}
		\leq Ct^{-\frac{n}{2p} +\epsilon}\norm{u_0}_{L^p}
\]
for $t\geq t_0>0$,
which finishes the proof of the theorem.
\end{proof}
We use Theorem \ref{thm : Derivative estimates introduction} to prove our first improved derivative estimates:

\begin{proof}[Proof of Proposition \ref{prop : better derivative estimates 0 introduction}]
By \eqref{commuting}, we have
\[
 \left(\D_m\right)^2\geq \nabla^*\nabla\geq0
\]
for $0\leq m\leq l-1$. Thus, these $\D_m$ have a trivial $L^2$-kernel and satisfy the assumptions of Theorem \ref{thm: main heat kernel estimate}.
By \eqref{commuting}, we may therefore apply Theorem \ref{thm_special_derivative_estimates} with
\[
	P = \n, \quad \Delta_V = \left(\D_{m-1}\right)^2, \quad \Delta_W = \left(\D_m\right)^2
\]
to get
\[
\norm{\nabla\circ e^{-t(\D_m)^2}}_{p \to p}\leq C t^{-\frac{1}{2}}
\]
for all $0 \leq m \leq l-1$ and $p\in (1,\infty)$.
Again, by \eqref{commuting}, we get
\[
	\nabla^k\circ e^{-t(\D_0)^2}
		=\nabla \circ e^{-\frac tk(\D_{k-1})^2}\circ\nabla\circ e^{-\frac tk(\D_{k-2})^2}\circ\ldots\circ \nabla\circ e^{-\frac tk(\D_0)^2},
\]
which implies that
\[
	\norm{\nabla^k\circ e^{-t(\D_0)^2}}_{p \to p}
		\leq \prod_{m=0}^{k-1}\norm{\nabla\circ e^{-\frac{t}{k}(\D_m)^2}}_{p \to p}\leq Ct^{-\frac{k}{2}}
\]
for $k \leq l$.
Similarly, for $k \geq l + 1$, we get
\begin{align*}
	\norm{\nabla^k\circ e^{-t(\D_0)^2}}_{p \to p}
		&\leq \norm{\nabla^{k-l}\circ e^{-\frac{t}{l+1}(\D_l)^2}}_{p \to p}\prod_{m=0}^{l-1}\norm{\nabla\circ e^{-\frac{t}{l+1}(\D_m)^2}}_{p \to p}\\
		&\leq Ct^{-\frac{l}{2}}\norm{\nabla^{k-l}\circ e^{-\frac{t}{l+1}(\D_l)^2}}_{p \to p}.
\end{align*}
We conclude that $e^{-t(\D_0)^2}$ satisfies (strong) derivative estimates of degree $l$ if
$e^{-t(\D_l)^2}$ satisfies (strong) derivative estimates of degree $0$. 
Since the latter conditions follow from Theorem \ref{thm : Derivative estimates introduction} and Remark \ref{rmk : Euclidean heat kernel estimates}, respectively. 
This finishes the proof of the theorem.
\end{proof}

\subsection{Improved derivative estimates} \label{subsec: improved derivative estimates}

We now turn to the proof of Theorem \ref{thm : better derivative estimates introduction}.
We will need the following consequence of Proposition \ref{prop: dirac Fredholm estimate}:
\begin{lemma} \label{le: semi-Fredholm}
Let $\D_V$ be a self-adjoint Dirac type operator acting on sections of a vector bundle $V$, which is asymptotic to a Euclidean Dirac operator in the sense of Definition \ref{def: asympt Eucl Dirac}. Let $E\subset V$ be a parallel subbundle of $V$.
Let $k \in \N_0$ and $p \in (1, \infty)$ and assume that 
\begin{equation}
	\de, \hdots, \de-(k-1) \in \R \backslash \E_1.
\end{equation}
The operator
\[
	\D_V^k|_E: W^{k, p}_\de(E) \to W^{0, p}_{\de-k}(V)
\]
is a semi-Fredholm operator, i.e.\ has finite dimensional kernel and closed range.
In particular, there is a closed subspace $X_\de \subset W^{k, p}_\de(E)$, such that
\begin{equation} \label{eq: X split}
	W^{k, p}_\de(E) 
		= \ker\left(\D_V^k|_E\right) \oplus X_\de
\end{equation}
and a $C > 0$, such that
\begin{equation} \label{eq: estimate inverse}
	\norm{u}_{k,p,\de} \leq C \norm{\D_V^k u}_{0,p,\de-k}
\end{equation}
for all $u \in X_\de$.
\end{lemma}
\begin{proof}
The proof is standard, given Proposition \ref{prop: dirac Fredholm estimate}. 
We first show that the kernel is finite dimensional. 
Let $(u_i)_i \subset \ker\left(\D_V^k|_E\right) \subset W^{k, p}_\de(V)$ be a bounded sequence.
Proposition \ref{prop: dirac Fredholm estimate} implies that 
\begin{equation} \label{eq: compact sequence}
	\norm{u_i}_{k,p,\de} \leq C \norm{u_i}_{L^p(B_R)}.
\end{equation}
By Rellich's lemma, the embedding $W^{k, p}_\de(V) \hookrightarrow L^p(B_R)$ is compact, hence there is a converging subsequence $u_{i_j} \to u$ in $L^p(B_R)$. 
The estimate \eqref{eq: compact sequence} implies that $u_{i_j} \to u$ in $W^{k,p}_\de(V)$ as well. 
In other words, any bounded sequence has a convergent subsequence.
It follows that the unit ball in $\ker\left(\D_V^k|_E\right)$ is a compact subset of $W^{k, p}_\de(V)$, implying that it is finite dimensional.
We now show that the range is closed, by first proving estimate \eqref{eq: estimate inverse}.
Since ${\ker\left(\D_V^k|_E\right) \subset W^{k, p}_\de(V)}$ is finite dimensional, there is a closed subspace $X_\de \subset W^{k, p}_\de(V)$, such that \eqref{eq: X split} holds.
Assume now that there is no constant $C > 0$, such that \eqref{eq: estimate inverse} holds.
Then there is a sequence $(u_i)_i \subset X_\de \subset W^{k, p}_\de(E)$, such that 
\begin{equation} \label{eq: range closed}
	\norm{u_i}_{k,p,\de} = 1, \quad \norm{\D_V^k u_i}_{0,p,\de-k} \to 0.
\end{equation}
Again, Rellich's lemma and Proposition \ref{prop: dirac Fredholm estimate} implies in this case the existence of a $u_{i_j}$, converging in $L^p(B_R)$. 
Inserting this in combination with \eqref{eq: range closed} into Proposition \ref{prop: dirac Fredholm estimate} implies that $u_{i_j}$ converges in $X_\de$.
Since $\norm{u_i}_{k,p,\de} = 1$, the limit is non-zero, i.e.\ $u_{i_j} \to u \neq 0$ in $X_\de \subset W^{k,p}_\de(E)$.
Continuity of $\D_V^k$ implies that $\D_V^ku_{i_j} \to \D_V^ku \neq 0$, which contradicts $\norm{\D_V^k u_i}_{0,p,\de-k} \to 0$.
This proves \eqref{eq: estimate inverse}.
If now $\D_V^ku_i \to f$, for $(u_i)_i \subset X_\de$, then \eqref{eq: estimate inverse} implies that $(u_i)_i \to u$ in $X_\de$.
Continuity of $\D_V^k$ implies that $\D_V^ku_i \to \D_V^ku = f$, showing that the range is closed.
\end{proof}

The following lemma allows us to eventually apply Lemma \ref{le: semi-Fredholm} to prove the improved derivative estimates:
\begin{lemma} \label{le: optimal up to l}
Assume the same as in Theorem \ref{thm : better derivative estimates introduction}.
Let $k \in \N_0$ and $p \in (1, \infty)$, such that 
\begin{align*}
	k - \frac np 
		\notin \N_0, \qquad
	k - \frac np
		< l,
\end{align*}
we have
\[
\norm{\nabla^ku}_{L^p}\leq C\norm{\D_V^ku}_{L^p},
\]
for all $u \in W^{k,p}_{k - \frac np}(E)$.
\end{lemma}
\begin{remark}
Note that the assumptions on $k$ and $p$ are equivalent to \eqref{eq: derivative k leq l}, \eqref{eq: k geq l plus 1 p small} and \eqref{eq: k leq l exceptional}.
\end{remark}
\begin{proof}
The proof is based on applying Lemma \ref{le: semi-Fredholm} with a fixed $\de := k - \frac np$.
Since $k - \frac np \notin \N_0$ and since $1 - \frac np > 1-n$, it is clear that $\de = k - \frac np, \hdots, \de-(k-1) = 1 - \frac np \notin \E_1$.
Lemma \ref{le: semi-Fredholm} therefore implies that
\[
	\D_V^k|_E: W^{k, p}_{k - \frac np}(E) \to W^{0, p}_{-\frac np}(V) = L^p(V)
\]
is a semi-Fredholm operator.
Assume that $u \in \ker\left(\D_V^k|_E\right) \subset W^{k,p}_\de(E)$.
By H\"{o}lder's inequality on $B_R$, it follows that $u \in L^q(B_R)$, for all $q \in [p, \infty)$.
Therefore Proposition \ref{prop: dirac Fredholm estimate} implies that $u \in W^{k,q}_\de(E)$, for all $q \in [p, \infty)$.
By \cite{Bartnik1986}*{Thm.\ 1.2}, it follows that $u \in o\left(r^\de\right)$.
If $k \leq l$, then $u \in o\left(r^\de\right) \subset o\left(r^k\right)$, so by the assumption in Theorem \ref{thm : better derivative estimates introduction}, we conclude that $\n^k u = 0$.
If instead $k \geq l$, then we use that $u \in o\left(r^\de\right) \subset o\left(r^l\right)$, which by the assumption in Theorem \ref{thm : better derivative estimates introduction} implies that $\n^l u = 0$ and hence $\n^k u = 0$.
To sum up, we have shown that
\[
	\n^k u = 0
\]
for all $u \in \ker\left(\D_V^k|_E\right)$.
This implies that
\[
	\n^k u = \n^k \mathrm{proj}_{X_\de}(u)
\]
for all $u \in W^{k, p}_\de(E)$, where $\mathrm{proj}_{X_\de}$ is the projection onto ${X_\de}$, given by the split \eqref{eq: X split}.
On the other hand, by construction of ${X_\de}$, we have
\[
	\D_V^k u = \D_V^k \mathrm{proj}_{X_\de}(u)
\]
for all $u \in W^{k, p}_\de(E)$.
Applying the estimate \eqref{eq: estimate inverse}, we get
\[
	\norm{\n^k u}_{L^p}  = \norm{\n^k \mathrm{proj}_{X_\de}(u)}_{L^p} \leq \norm{\mathrm{proj}_{X_\de}(u)}_{k,p,k-\frac np} \leq \norm{\D_V^k \mathrm{proj}_{X_\de}(u)}_{L^p} = \norm{\D_V^k u}_{L^p}
\]
for all $u \in W^{k, p}_\de(E)$.
\end{proof}

The previous lemma takes care of the case $	k - \frac np < l$.
We now provide a corresponding estimate when $k - \frac np \geq l$:

\begin{lemma} \label{le: large k estimate}
Assume the same as in Theorem \ref{thm : better derivative estimates introduction}.
Let $k \in \N_0$ and $p \in (1, \infty)$, such that 
\[
	k - \frac np \geq l.
\]
Then for any any $q \in (p, \infty)$ and
$	\eta \in \left(l - \frac nq, l\right) \cap \left(l-1, l \right)$,
we have
\[
	\norm{\n^k u}_{L^p(M)} \leq C \left(\norm{\D_V^k u}_{L^p(M)} + \norm{\D_V^l u}_{L^q(M)} \right)
\]
for all $u \in W^{k,p}_\eta(E)$.
\end{lemma}
\begin{remark}
Note that the conditions on $p$ and $k$ are equivalent to \eqref{eq: k geq l plus 1 p large}.
\end{remark}

When applying this lemma to prove Theorem \ref{thm : better derivative estimates introduction}, we will assume that $q$ is very large, so we think of $\eta$ as being very close to $l$.
\begin{proof}
Note that $\eta, \hdots, \eta - (l-1) \notin \E_1$.
We consider
\[
	\D_V^l|_E: W^{l, p}_\eta(E) \to W^{0, p}_{\eta-l}(V).
\]
By Lemma \ref{le: semi-Fredholm} and arguing as in the proof of Lemma \ref{le: optimal up to l}, we know that
\[
	\ker \left(\D_V^l|_E\right) \subset o\left(r^l\right).
\]
By the assumption in Theorem \ref{thm : better derivative estimates introduction}, we conclude that 
\[
	\n^l u = 0,
\]
for all $u \in \ker\left(\D_V^l|_E\right)$.
This implies that
\[
	\n^l u = \n^l \mathrm{proj}_{X_\eta}(u), \quad \D_V^l u = \D_V^l \mathrm{proj}_{X_\eta}(u)
\]
for all $u \in W^{l, p}_\eta(E)$, where $\mathrm{proj}_{X_\eta}$ is the projection onto $X_\eta$, given by 
\[
	W^{l, p}_\eta(E) = \ker\left(\D_V^l\right) \oplus X_\eta.
\]
Since $k \geq l+1$, we also get
\[
	\n^k u = \n^k \mathrm{proj}_{X_\eta}(u), \quad \D_V^k u = \D_V^k \mathrm{proj}_{X_\eta}(u)
\]
for all $u \in W^{k, p}_\eta(E) \subset W^{l, p}_\eta(E)$.
Using this, and applying Proposition \ref{prop: dirac non Fredholm estimate} with $\de := k - \frac np$, we get
\begin{align*}
	\norm{\n^k u}_{L^p}
		&= \norm{\n^k \mathrm{proj}_{X_\eta}(u)}_{L^p}
		\leq \norm{\mathrm{proj}_{X_\eta}(u)}_{k,p,\de} \\
		&\leq C \left(\norm{\D_V^k \mathrm{proj}_{X_\eta}(u)}_{0,p,\de-k} + \norm{\mathrm{proj}_{X_\eta}(u)}_{0, p, \de}\right) \\
		&\leq C \left(\norm{\D_V^k u}_{0,p,\de-k} + \norm{\mathrm{proj}_{X_\eta}(u)}_{0, p, \de}\right).
\end{align*}
It remains to estimate the second term on the right hand side.
For this, we note that
\[
	\de > \eta, \qquad \eta - l > - \frac nq.
\]
By Proposition \ref{prop: dirac Fredholm estimate}, we get the estimate
\begin{align*}
	\norm{\mathrm{proj}_{X_\eta}(u)}_{0, p, \de} 
		&\leq C\norm{\mathrm{proj}_{X_\eta}(u)}_{l, p, \eta} 
		\leq C\norm{\D_V^l \mathrm{proj}_{X_\eta}(u)}_{0, p, \eta - l} \\
		&\leq C\norm{\D_V^l \mathrm{proj}_{X_\eta}(u)}_{0, q, -\frac nq} 
		= C\norm{\D^l_V u}_{L^q}
\end{align*}
for all $u \in W^{k,p}_\eta(E)$.
This finishes the proof.
\end{proof}

We are finally in shape to apply these estimates to prove Theorem \ref{thm : better derivative estimates introduction}:

\begin{proof}[Proof of Theorem \ref{thm : better derivative estimates introduction}]
Assume that $u_0 \in L^p(E)$ is $L^2$-orthogonal to $\ker_{L^2}\left(\D_V^2 \right)$, and denote $u = e^{-t\D_V^2}u_0$.
Since $\D_V^2$ maps sections in $E$ to sections in $E$, it follows that $u$ is a section in $E$
By Theorem \ref{thm: main heat kernel estimate}, we know that
\[
	u \in L^p(E) = W^{0,p}_{-\frac np}(E) \subset W^{0,p}_{k-\frac np}(E)
\]
and by Lemma \ref{lem : higher derivatives}, we know that $\D_V^ku \in L^p(E) = W^{0,p}_{-\frac np}(E)$.
Proposition \ref{prop: dirac non Fredholm estimate} implies therefore that $u \in W^{k,p}_{k-\frac np}(E)$, which allows us to apply the previous lemmas in what comes.
For any $k \in \N_0$ and $p \in (1, \infty)$, such that 
\[
	k - \frac np \notin \N_0, \quad k - \frac np < l,
\]
Lemma \ref{le: optimal up to l} and Lemma \ref{lem : higher derivatives} implies that 
\[
	\norm{\nabla^ku}_{L^p}
		\leq C\norm{\D_V^ku}_{L^p} 
		\leq Ct^{-\frac k2} \norm{u_0}_{L^p},
\]
as claimed in the theorem.
Now assume that 
\[
	k - \frac np \in \N_0, \quad k - \frac np < l.
\]
This allows us to choose an $m \in \N_0$, such that
\[
	m - \frac np = l.
\]
For a given $\e > 0$, let $q \in (p, \infty)$ be large enough such that $\frac n{2q}\leq \frac{\e m}k$.
We now combine the Gagliardo-Nirenberg interpolation inequality (see e.g.\ \cites{FFRS,Leo17}) with Theorem \ref{thm: main heat kernel estimate}, Lemma \ref{le: large k estimate} and Lemma \ref{lem : higher derivatives} to get
\begin{align*}
	\norm{\nabla^ku}_{L^p}
		&\leq C \norm{\nabla^mu}_{L^p}^{\frac km} \norm{u}_{L^p}^{1-\frac km}
			\leq C \left(\norm{\D_V^m u}_{L^p} + \norm{\D_V^lu}_{L^q} \right)^{\frac km}\norm{u_0}_{L^p}^{1-\frac km} \\*
		&\leq C \left(t^{-\frac m2} + t^{- \frac n2 \left(\frac1p - \frac1q\right) - \frac l2 }\right)^{\frac km}\norm{u_0}_{L^p} 
			\leq C \left(t^{-\frac m2} + t^{- \frac12\left(\frac n{p} + l\right) + \frac{\e m}k}\right)^{\frac km}\norm{u_0}_{L^p} \\*
		&=Ct^{-\frac{k}{2} + \epsilon}\norm{u_0}_{L^p}
\end{align*}
for $t \geq t_0$, as claimed in the theorem.
We finally consider the case when $k \in \N_0$ and $p \in (1, \infty)$, such that 
\[
	k - \frac np \geq l.
\]
For a given $\e > 0$, let $q \in (p, \infty)$ be large enough such that $\frac n{2q}\leq \e$.
Now, Lemma \ref{le: large k estimate} and Lemma \ref{lem : higher derivatives} imply that 
\begin{align*}
	\norm{\n^k u}_{L^p(M)} 
		&\leq C \left(\norm{\D_V^k u}_{L^p(M)} + \norm{\D_V^l u}_{L^q(M)} \right)
			\leq C \left(t^{-\frac{k}{2}}+t^{-\frac{n}{2}(\frac{1}{p}-\frac{1}{q})-\frac{l}{2}}\right)\norm{u_0}_{L^p} \\*	
		 &\leq C t^{-\frac{n}{2p}-\frac{l}{2}+\epsilon}\norm{u_0}_{L^p}
\end{align*}
for $t \geq t_0$, which completes the proof.
\end{proof}

For future applications, we note the following corollary:

\begin{cor}\label{Cor : improved derivative estimates}
Suppose that we have a parallel subbundle $E \subset V$ which is invariant under $\D_V^2$ and assume that $\Delta_E := \D_V^2|_E$ is of the form
\[
	\Delta_E = \nabla^*\nabla+\mathcal{R},
\]
with a non-negative symmetric endomorphism $\mathcal{R}\in C^{\infty}(\mathrm{End}(E))$.
Then, if $(M,g)$ has only one end, $e^{-t\Delta_E}$ satisfies Euclidean heat kernel estimates and weak derivative estimates of degree $1$.
\end{cor}

\begin{proof}
Let $u\in C^{\infty}(E)$ with $u=o(r)$ be such that $\D_V u=0$. 
Then $\Delta_E u=0$. 
Since there are no exceptional values between $0$ and $1$, $u$ is bounded. 
Because
\[
\Delta |u|^2=-|\nabla u|^2-\langle \mathcal{R}u,u\rangle\leq 0
\]
and $(M,g)$ has only one end, the maximum principle implies equality on the right hand side. 
Therefore, $\nabla u \equiv 0$. 
Since this implies $\ker_{L^2}(\Delta_E)=\left\{0\right\}$, the first assertion follows from Remark \ref{rmk : Euclidean heat kernel estimates}. 
The second assertion follows from Theorem \ref{thm : better derivative estimates introduction} above.
\end{proof}

We finish the section by proving Corollary \ref{cor : better derivative estimates}:

\begin{proof}[Proof of Corollary \ref{cor : better derivative estimates}]
We need to show that the assumptions in the corollary imply the assumptions in Theorem \ref{thm : better derivative estimates introduction}. In this situation, $V=\Delta M$ will be the exterior algebra, $E=M\times\R$ the trivial line bundle and $\mathcal{D}_V=d+d^*$ the Hodge de Rham operator on $\Delta M$.
In fact, we will prove the following stronger implication, which clearly implies the assumptions of Theorem \ref{thm : better derivative estimates introduction}:
For every $k \in \{1, \hdots, l\}$, every solution to $(d+d^*)^k u = 0$ with $u \in o(r^k)$ is constant.
For $k = 1, 2$, this trivially follows from the assumptions in the corollary.
We therefore proceed by induction.
Assume that the implication holds for $k-2$, where $k \leq l$ and let $u \in o\left(r^k\right)$ satisfy $\left(d + d^*\right)^k u = 0$.
By a standard argument, similar to (but simpler than) the proof of Proposition \ref{prop : decay harmonic forms} below, one shows that $u \in \O_\infty\left(r^{k-1}\right)$.
We conclude that $v := \Delta u$ satisfies $\left(d + d^*\right)^{k-2} v = 0$, which by the induction assumption implies that $v$ is constant.
Since the statement is known when $k = 2$, we conclude that $u$ is constant.
This finishes the proof.
\end{proof}

\section{Harmonic forms on ALE manifolds}\label{sec : decay harmonic forms}
\subsection{On the model cone}
\label{subsec : decay harmonic forms}
\noindent
We start by analyzing the decay of harmonic functions on $\R^n_*$.
\begin{lemma}\label{le: harmonic functions on flat space}
Let $f$ be a harmonic function on $\R^n_*$, which decays at infinity, i.e.  $f\to 0$ as $r\to\infty$. Then,
\[
f(x)=Ar^{2-n}+\langle B,x\rangle r^{-n}+g(x),
\]
where $A\in\R$, $B\in\R^n$ and $g\in \O_{\infty}(r^{-n})$.
\end{lemma}
\begin{proof}
By \cite{ABR01}, c.f. also \cite{Chen20}*{Lem.\ 3.1}, we have
\[
f(x)=f(r,\theta)=\sum_{i=0}^{\infty} C_i\cdot \varphi_i(\theta)\cdot r^{2-n-i},
\]
where $C_i\in\R$, $(r,\theta$ are polar coordinates and $\varphi_i\in C^{\infty}(S^{n-1})$ is a normalized eigenfunction on the sphere to the i-th eigenvalue. Clearly, $\phi_0$ is constant. 
The eigenfunction $C_1\phi_1$ is the restriction of a linear function $x\to \langle B,x\rangle$ to $S^{n-1}$, see \cite{BGM71}*{p.\ 159--161}. 
Therefore,
\[
f(x)=Ar^{2-n}+\langle B,x\rangle r^{-n}+g(x),
\]
where the remainder term
\[
g(x)=g(r,\theta)=\sum_{i=2}^{\infty} C_i\cdot \varphi_i(\theta)\cdot r^{2-n-i}
\]
is again a harmonic function. 
Standard arguments from elliptic theory imply that this series converges in all derivatives. 
Therefore, $g\in \O_{\infty}(r^{-n})$, which finishes the proof of the lemma.
\end{proof}

We continue by analyzing the decay of harmonic forms on $\R^n_*$:

\begin{lemma} \label{le: harmonic forms on flat space}
Let $\omega$ be differential form of degree $k$ on $\R^n_*$, which satisfies $d\omega=0$ and $d^*\omega=0$ and decays at infinity, i.e.  $|\omega|\to 0$ as $r\to\infty$.
\begin{itemize}
	\item If $k=0,n$, we have $\omega=0$.
	\item If $k=1$, we have $\omega=\lambda\cdot d\Phi + \O_\infty(r^{-n})$, where $\Phi=r^{2-n}$ is the fundamental solution.
	\item If $k=n-1$, we have $\omega=\lambda\cdot d^*(\Phi\cdot \mathrm{dvol})+ \O_\infty(r^{-n})$.
	\item If $2\leq k\leq n-2$, we have $\omega=\O_\infty(r^{-n})$.
\end{itemize}
\end{lemma}

\begin{proof}
In the case $k = 0$ or $k = n$, the statement is trivial.
Therefore, we may assume from now on that $1 \leq k \leq n-1$. 
We write $\omega$ in the standard coordinates as
\begin{align*}
	\omega
		&= \sum_{i_1,i_2,\ldots,i_k=1}^n\omega_{i_1\ldots i_k}dx^{i_1} \otimes \ldots \otimes dx^{i_k},
\end{align*}
where $\omega_{i_1\ldots i_k}=\omega(\partial_{i_1},\ldots,\partial_{i_l})$ are totally anti-symmetric in all indices.
Since $\Delta=dd^*+d^*d$, we have $\omega\in\ker\Delta$. 
This implies that each component function $\o_{i_1 \ldots i_k}$ is a also harmonic function on $\R^n \backslash \{0\}$. 
By Lemma \ref{le: harmonic functions on flat space}, we get
\begin{align*}
	\omega_{i_1\ldots i_k}
		=\o^{(0)}_{i_1\ldots i_k}r^{2-n} + \sum_{j=1}^n \o^{(1)}_{j, i_1\ldots i_k} x_jr^{-n} + \O_{\infty}(r^{-n}),
\end{align*}
where for each fixed $j$, the coefficients $\o^{(0)}_{i_1\ldots i_k}$ and $\o^{(1)}_{j, i_1\ldots i_k}$, are totally anti-symmetric in $i_1,\ldots,i_k$.
By the local expression for $0 = d^* \o$, we get
\begin{equation}\label{eq : coclosed form}
\begin{split}
	0 
		&= (d^*\omega)_{i_2\ldots i_k}(x) 
			=-\sum_{l=1}^n\partial_l\omega_{l i_2\ldots i_k}(x) \\
		&= - \sum_{l=1}^n\o^{(0)}_{l i_2 \ldots i_k}x_l(2-n)r^{-n} - \sum_{j, l=1}^n\o^{(1)}_{j, l i_2\ldots i_k} \left(\de_{jl} - n \frac{x_j x_l}{r^2}\right)r^{-n}+\O_\infty(r^{-n-1})
\end{split}
\end{equation}
for all $x\in\R^n$.
Due to the different fall-off behavior, the two terms involving $\o^{(0)}_{l i_2 \ldots i_k}$ and $\o^{(1)}_{j, l i_2 \ldots i_k}$ terms have to vanish separately, i.e.\
\begin{align*}
	 \sum_{l=1}^n\o^{(0)}_{l i_2 \ldots i_k}x_l
	 	= 0, \qquad
	 \sum_{j, l=1}^n\o^{(1)}_{j, l i_2\ldots i_k} \left(\de_{jl} - n \frac{x_j x_l}{r^2}\right) 
	 	&= 0
\end{align*} 
for all $x \in \R^n$.
By inserting the unit vectors $x = (0, \ldots,1, \ldots, 0) \in \R^n$ into the first equation, one successively concludes that
\begin{align*}
	\o^{(0)}_{l i_2 \ldots i_k} = 0.
\end{align*}

For the next term, fix $i_2, \ldots, i_k$ and define the quadratic form $C_{jl} := \o^{(1)}_{j,l i_2\ldots i_k}$.
The second equation is equivalent to
\begin{align*}
	 0
	 	= \sum_{j, l=1}^n\o^{(1)}_{j, l i_2\ldots i_k} \left(\frac{\de_{jl}\ldr{x,x}}n - x_j x_l\right) 
	 	= \frac1n \tr(C) \ldr{x, x} - C(x, x),
\end{align*}
for all $x \in \R^n$, where we have used that $r^2 = \ldr{x, x}$.
This implies that the trace-free part of the symmetric part of $C$ vanishes. 
In other words, we have proven that
\begin{equation} \label{eq: d star equation}
	\o^{(1)}_{j,l i_2\ldots i_k} + \o^{(1)}_{l,j i_2\ldots i_k} = \de_{jl} \eta_{i_2 \ldots i_k}
\end{equation}
for all $j, l$ and coefficients $\eta_{i_2 \ldots i_k}$ which are totally antisymmetric in $i_2,\ldots,i_k$.

We first treat the case $k=1$. 
In this case, the equation \eqref{eq: d star equation} reads
\begin{equation} \label{eq: k 1}
	\o^{(1)}_{j,l} + \o^{(1)}_{l,j} = 2\lambda \de_{jl},
\end{equation}
for some constant $\lambda \in \R$.
We now use that $d\o = 0$ and conclude that
\[
	0
		= \d_k \o_l - \d_l \o_k 
		= \sum_{j = 1}^n \o_{j,l}^{(1)} \left( \de_{jk} - n \frac{x_j x_k}{r^2} \right)r^{-n} - \sum_{j = 1}^n \o_{j,k}^{(1)} \left( \de_{jl} - n \frac{x_j x_l}{r^2} \right) r^{-n} + \O(r^{n-1}).
\]
Inserting $x$, such that $x_i = \de_{il}$, we get
\begin{align*}
	0 
		&= \o_{k,l}^{(1)} + (n-1) \o_{l,k}^{(1)}.
\end{align*}
Combining this with \eqref{eq: k 1}, we conclude that
\[
	\o_{k,l}^{(1)} = \lambda \de_{kl}.
\]
Inserting this in the expansion of harmonic functions, one checks that
\[
	\o = \lambda \cdot d\Phi + \O_\infty(r^{-n}),
\]
if $k = 1$.
The statement for $k = n-1$ follows by Hodge duality.

We now turn to the case when $2 \leq k \leq n-2$. 
Since $k \geq 2$, we may insert $j = l = i_2$ into \eqref{eq: d star equation} and conclude that 
\[
	\eta_{i_2 \ldots i_k} = \de_{i_2 i_2} \eta_{i_2 \ldots i_k} = 2 \o^{(1)}_{i_2, i_2 i_2 \ldots i_k} = 0
\]
and therefore,
\begin{align*}
	\o^{(1)}_{j, l i_2\ldots i_k} + \o^{(1)}_{l, j i_2\ldots i_k}=0,
\end{align*}
which means that
$	\o^{(1)}_{j, i_1 \ldots i_k}$
is totally anti-symmetric in all indices, including $j$.
In order to show that in fact $\o^{(1)}_{j, i_1 \ldots i_k}$ vanishes, we use our assumption $d\o = 0$.
This is locally given by
\begin{align*}
	0 = (d\omega)_{i_1\ldots i_{k+1}} = \frac{1}{k+1}\sum_{l=1}^{k+1}(-1)^{l+1}\partial_{i_l}\omega_{i_1,\ldots,\widehat{i_l},\ldots i_{k+1}}.
\end{align*}
Inserting the expansion for $r \to 0$, we again conclude that term separately must vanish, which implies that
\begin{equation} \label{eq: closed form}
	0 
		= \sum_{l=1}^{k+1}(-1)^{l+1}\partial_{i_l}\left(\sum_{j=1}^n \o^{(1)}_{j, i_1\ldots \hat{i_l} \ldots i_{k+1}} x_jr^{-n}\right) 
			= \sum_{j = 1}^n \sum_{l=1}^{k+1}(-1)^{l+1}\o^{(1)}_{j, i_1 \ldots \widehat{i_l}\ldots i_{k+1}}\left(\de_{ji_l} - n \frac{x_j x_{i_l}}{r^2}\right)r^{-n}.
\end{equation}
We evaluate this expression for a certain choice of $x$.
Since $k \leq n-2$, there is an 
\[
	m \in \left\{1,\ldots,n\right\}\setminus \left\{i_1,\ldots i_{k+1}\right\}.
\]
Evaluating \eqref{eq: closed form} at the $m$-th unit vector, i.e.\ at $x_i = \de_{im}$, note that all $x_{i_j} = 0$ and hence
\begin{align*}
	0 
		&= \sum_{j = 1}^n \sum_{l=1}^{k+1}(-1)^{l+1}\o^{(1)}_{j, i_1 \ldots \widehat{i_l},\ldots i_{k+1}}\left(\de_{ji_l} - n \frac{x_j x_{i_l}}{r^2}\right) 
			= \sum_{l=1}^{k+1}(-1)^{l+1}\o^{(1)}_{i_l, i_1 \ldots \widehat{i_l},\ldots i_{k+1}} \\
		&= (k+1) \omega^{(1)}_{i_1, i_2 \ldots i_{k+1}},
\end{align*}
where we have used that $\omega^{(1)}_{i_1, i_2 \ldots i_{k+1}}$ is totally anti-symmetric in all indices, proven above.
We conclude that 
\[
	\omega^{(1)}_{i_1, i_2 \ldots i_{k+1}} = 0
\]
for any $i_1, \ldots, i_{k+1}$.
Inserting this in the asymptotic expansion above, we conclude that
\[
	\omega = \O_\infty(r^{-n}),
\]
as claimed, if $2 \leq k \leq n-2$.
\end{proof}

\noindent
\subsection{On ALE manifolds}
Let $\Delta_H$ the Hodge-Laplace operator.
\begin{prop}\label{prop : decay harmonic forms}
Let $(M^n,g)$ be an ALE manifold.
Let $\omega$ be harmonic differential form on $M$, i.e.\
\[
	\Delta_H \o = 0
\]
and suppose that $u\in L^p$ for some $p\in(1,\infty)$.
Then, $d\omega$ and $d^*\omega$ are both vanishing and $\o = \O_\infty(r^{1-n})$. 
Moreoever, if $\o$ is not a one-form or an $n-1$-form, then $\o = \O_\infty(r^{-n})$.
\end{prop}
\begin{proof}
Because $\Delta_H \o = 0$ and $u\in L^p$ for some $p\in(1,\infty)$, elliptic regularity implies that $\o\to 0$ at infinity. Standard arguments (see e.g.\ \cites{Bartnik1986,pacini}) imply that
\[
	\o \in \O_{\infty}(r^{2-n}).
\]
In particular,
\[
d\o\in \O_{\infty}(r^{1-n})\subset L^2,\qquad d^*\o\in \O_{\infty}(r^{1-n})\subset L^2,
\] 
since $n\geq 3$.
Recalling that $\Delta_H \o=d^*d\o+dd^*\o = 0$,
integration by parts over a large ball $B_R$ and letting $R\to \infty$ shows that $d\o=0$ and $d^*\o=0$. 
The operator
\[
	\D := d + d^*
\]
is asymptotic to the Dirac type operator
\[
	\D_{\R^n} := d + d^*_{\R^n}.
\]
 Following \cite{Bartnik1986} and \cite{BKN89}, we may split $\o=\alpha+\beta$ at infinity, where $\D_{\R^n}\alpha=0$. 
From Lemma \ref{le: harmonic forms on flat space}, we get
\[
	\alpha \in \O_{\infty}(r^{1-n}),
\]
and if $\alpha$ is not a 1-form or an $n-1$-form,
\[
	\alpha \in \O_{\infty}(r^{-n}).
\]
Standard bootstrapping arguments in weighted spaces (as  carried out e.g.\ in the proof of \cite{DK17}*{Thm.\ 2.7}) show that $\beta$ has to decay faster than $\alpha$ which implies the desired result for $\o$.
\end{proof}

\section{Heat flows on general ALE manifolds}\label{sec : geometric operators 1}

\subsection{The Hodge-Laplacian on the exterior algebra}
The following theorem is Corollary \ref{cor : heat flow hodge laplace}:
\begin{thm}[Heat kernel and derivative estimates for the Hodge-Laplacian]\label{thm : heat flow Hodge Laplace} Let $(M,g)$ be an ALE manifold and suppose that $\mathcal{H}_1(M)=0$. Then, $e^{-t\Delta_H}$ satisfies almost Euclidean heat kernel estimates and derivative estimates of degree $0$.
\end{thm}
\begin{proof}
Let $\omega\in \ker_{L^2}(\Delta_H)$. 
By assumption, $\omega$ is not a one-form and by Hodge duality, it is not an $n-1$-form either.
By Proposition \ref{prop : decay harmonic forms}, this implies $w\in\O_{\infty}(r^{-n})$. The assertion then follows from Theorem \ref{thm: main heat kernel estimate} and Theorem \ref{thm : Derivative estimates introduction}.
\end{proof}
\begin{remark}
In the next two subsections, we demonstrate that these results can be substantially improved if we restrict to functions and one-forms. By Hodge duality, the same estimates also hold for $n$-forms and $(n-1)$-forms, respectively.
\end{remark}
\subsection{The Hodge-Laplacian on one-forms}Here, we are going to prove Corollary \ref{cor : heat flow one forms} (i) and (ii) which are Theorem \ref{thm : heat flow one forms} and Theorem \ref{thm : strong order 1 one forms}, respectively. Recall that $\Delta_{H_1}$ denotes the Hodge Laplacian on $T^*M$.
For convenience, we will in the following use the notation $\Delta_{H_1}:=\Delta_{H}|_{T^*M}$.
\begin{thm}\label{thm : heat flow one forms}
Let $(M,g)$ be an ALE manifold and suppose that $\mathcal{H}_1(M)=0$. Then, $e^{-t\Delta_{H_1}}$ satisfies Euclidean heat kernel estimates and strong derivative estimates of degree $0$.
\end{thm}
\begin{proof}
Because
\begin{align}\label{eq : Bochner one form Laplacian}
\Delta_{H_1}=\nabla^*\nabla+\Ric
\end{align}
and $\ker_{L^2}\Delta_{H_1}=\mathcal{H}_1(M)=0$, as the first assertion follows from work by Devyver \cite{Devyver2014}*{Thm.\ 3}, we also pointed out in Remark \ref{rmk : Euclidean heat kernel estimates}.
Furthermore, we already know that we have derivative estimates of degree zero. 
Thus it remains to show that
\[
\norm{\nabla^k\circ e^{-t\Delta_H}\omega}_{L^p}\leq C t^{-\frac{n}{2p}}\norm{\omega}_{L^p}
\]
for all $p\in (\frac{n}{k},\infty)$ and $k\geq 1$.
Let $\omega_t=e^{-t\Delta_{H_1}}\omega$.
Corollary \ref{cor: elliptic dirac gradient estimate 2}, Theorem \ref{thm_special_derivative_estimates} and Euclidean heat kernel estimates then imply that
\[
	\norm{\nabla^k\omega_t}_{L^p}
		\leq C(\norm{(d+d^*)^k\omega}_{L^p}+\norm{\omega_t}_{L^{\infty}})
		\leq C\left(t^{-\frac{k}{2}}+t^{-\frac{n}{2p}}\right)\norm{\omega_0}_{L^p}
		\leq C t^{-\frac{n}{2p}}\norm{\omega_0}_{L^p}
\]
for all $t\geq t_0$ which finishes the proof.
\end{proof}
\begin{thm}\label{thm : strong order 1 one forms}
Let $(M,g)$ be an ALE manifold with non-negative Ricci curvature. Then $e^{-t\Delta_{H_1}}$ satisfies weak derivative estimates of degree $1$.
\end{thm}
\begin{proof}
First let us see that $(M,g)$ has only one end: Suppose that there were at least two, then $(M,g)$ contains a line, i.e.\ a geodesic that is minimizing between each of its points. Because $\Ric\geq0$, the Cheeger-Gromoll splitting theorem implies that $(M,g)$ splits isometrically as a product $(\R\times N,dr^2+h)$. However, if $(N,h)\neq (\R^{n-1},g_{eucl})$, curvature would not fall of at infinity which contradicts the assumption that $(M,g)$ is ALE.\\
Due to \eqref{eq : Bochner one form Laplacian},
the assertion follows from Corollary \ref{Cor : improved derivative estimates}, because we assumed  $\Ric\geq0$.
\end{proof}
\begin{remark}
If $(M,g)$ is Ricci-flat and obtains a parallel spinor, the latter result can be slightly improved. This will be discussed in Subsection \ref{subsec : better 1-forms} below.
\end{remark}
\subsection{The Laplace-Beltrami operator on functions}
In this subsection, we are going to prove the two assertions of Corollary \ref{cor : heat flow functions}.
\begin{thm}[Derivative estimates]\label{thm : Derivative estimates functions ALE}  Let $(M^n,g)$ be an ALE manifold with $\mathcal{H}_1(M)=0$ and $\Delta$ be its Laplace-Beltrami operator. Then, $e^{-t\Delta}$ satisfies Euclidean heat kernel estimates and strong derivative estimates of degree $1$.
\end{thm}
\begin{proof}
The first assertion follows from Remark \ref{rmk : Euclidean heat kernel estimates}.
Because the differential $\nabla=d: C^{\infty}(M)\to C^{\infty}(M,T^*M)
$
satisfies
\[d\circ\Delta=\Delta_{H_1}\circ d,\qquad 
d^*\circ d=\Delta,\qquad d\circ d^*\leq \Delta_{H_1},
\]
the second assertion now follows from Proposition \ref{prop : better derivative estimates 0 introduction} and Theorem \ref{thm : heat flow one forms}.
\end{proof}
\begin{thm}[Derivative estimates]  Let $(M^n,g)$ be an ALE manifold with $\mathcal{H}_1(M)=0$, which is not AE.  Let $\Delta$ be its Laplace-Beltrami operator. Then, $e^{-t\Delta}$ satisfies derivative estimates of degree $2$.
\end{thm}
\begin{proof}
Let us recall a few more facts for the Laplace operator, see e.g\ \cite{pacini}*{Sec.\ 9} for details. On an asymptotically conical manifold $(M^n,g)$ with link $(L,g_L)$, 
\begin{align*}
\Delta:W^{k+2,p}_{\delta}(M)\to W^{k,p}_{\delta-2}(M)
\end{align*}
is Fredholm for $\delta\in \R\setminus{\mathcal{E}_{L}}$, where the exceptional set is given by
\[
\mathcal{E}_{L}=\left\{-\frac{n-2}{2}\pm\sqrt{\frac{(n-2)^2}{4}+\lambda}\mid \lambda \in \mathrm{Spec}(L)\right\}. 
\]
If $(M^n,g)$ is ALE, i.e.\ asymptotically conical with link $(L,g_L)=(S^{n-1}/\Gamma,g_{st})$ and we have
\[
\mathrm{Spec}(S^n/\Gamma,g_{st})\subset\left\{k(n+k-2)\mid k\in\N\right\}.
\]
However, if $\Gamma$ is nontrivial, the well-known Lichnerowicz-Obata eigenvalue estimate (see \cite{Obata1962}) implies that $(n-1)\notin \mathrm{Spec}(S^n/\Gamma,g_{st})$. For this reason, we have
\[
\mathcal{E}_{L}\subset \mathcal{E}_2\setminus \left\{1-n,1\right\} 
\]
in this case. In particular any harmonic $u\in C^{\infty}(M)$ with $u=o(r^{2})$ at infinity is bounded since there are no exceptional values in $(0,2)$. Furthermore, $u$ is constant: We have $\Delta_{H_1} du= d\Delta u$ and due to elliptic regularity, $du=\O_{\infty}(r^{-1})$. Proposition \ref{prop : decay harmonic forms} then implies that $du=\O(r^{1-n})$, so that 
\[du\in \ker_{L^2}(\Delta_{H_1})= \mathcal{H}^1(M)=\left\{0\right\}.\] Summing up, we have shown the implication
\[
	(d+d^*)^ku=0, \quad u=o(r^{k}) \quad \Rightarrow \quad \nabla^k u=0
\]
for $k=1,2$, which is condition \eqref{eq: weak implication derivatives}.
Theorem \ref{thm : better derivative estimates introduction} implies that $e^{-t\Delta}$ has weak derivative estimates of degree $2$. In particular we have
\begin{align}\label{eq : optimal derivative decay functions}
\norm{\nabla^k u}_{L^p}\leq C t^{-\frac{k}{2}}\norm{u_0}_{L^p}
\end{align}
for $u = e^{-t\Delta} u_0$ and pairs $(p,k)$ satisfying
\begin{align*}
	p
		&\in (1,n)\cup (n,\infty), \qquad k=1, \\
	p
		&\in \left(1,\frac n2\right) \cup \left(\frac n2,n\right) \cup (n,\infty), \qquad k=2, \\
	p
		&\in \left(1,\frac nk\right) \cup \left(\frac nk,\frac n{k-1} \right) \cup \left(\frac n{k-1}, \frac n{k-2}\right), \qquad k>2.
\end{align*}
In order to finish the proof of the theorem, we have to establish \eqref{eq : optimal derivative decay functions} also for the gap points in these intervals. In fact, since $1$ is no longer exceptional, the proof of Theorem \ref{thm : better derivative estimates introduction} also works for the non-exceptional value $\delta=k-\frac{n}{p}=1$. Hence we get \eqref{eq : optimal derivative decay functions} also for $k\geq 2$ and $p=\frac{n}{k-1}$. For the remaining cases $k\geq 1$ and $p=\frac{n}{k}$ interpolation shows that
\[
	\norm{\nabla^k u}_{L^p}
		\leq \norm{\nabla^{k-1} u}_{L^p}^{1/2}\norm{\nabla^{k+1}u}_{L^p}^{1/2}
		\leq C t^{-\frac k2}\norm{u_0}_{L^p},
\]
where the second inequality follows from cases which are already covered.
\end{proof}
\subsection{The classical Dirac operator}
Let $(M^n,g)$ be an ALE spin manifold and $\D_S$ be the classical Dirac operator acting on sections of the spinor bundle $S$. The following theorem is the first part of Corollary \ref{cor : heat flow dirac operator}.
\begin{thm}[Derivative estimates]\label{thm : heat kernel Dirac} Let $(M,g)$ be an ALE spin manifold with non-negative scalar curvature and one end. Then, $e^{-t(\D_S)^2}$ satisfies almost euclidean heat kernel estimates and weak derivative estimates of degree $1$.
\end{thm}
\begin{proof}
Due to the well known Weitzenb\"{o}ck formula
\begin{align}\label{eq : Bochner Dirac}
(\D_S)^2=\nabla^*\nabla+\frac{1}{4}\scal,
\end{align}
the result follows from the assumption $\scal\geq 0$ and Corollary \ref{Cor : improved derivative estimates}. 
\end{proof}

\section{Heat flows on ALE manifolds with a parallel spinor}\label{sec : geometric operators 2}
Throughout this section, we assume that $(M,g)$ is a simply-connected ALE spin manifold which admits a parallel spinor, i.e.\ a section $\sigma\neq 0$ of the spinor bundle such that $\nabla\sigma=0$. For convenience, we assume that $(M,g)\neq (\R^n,g_{eucl})$.
 These assumptions have various geometric consequences:
\begin{itemize}
\item $(M,g)$ is Ricci-flat.
\item $(M,g)$ has irreducible holonomy: Otherwise, $(M,g)=(\R\times N,dr^2+h)$ but this contradicts the assumption of the manifold to be ALE unless it is flat. Consequently,
\begin{align}\label{holonomy}
\mathrm{Hol}(M,g)\in\left\{\mathrm{SU}(n/2),\mathrm{Sp}(n/4),\mathrm{Spin}(7)\right\}.
\end{align}
\item $M$ is even-dimensional (therefore, we excluded the case of holonomy $G_2$): If the dimension was odd, then the group $\Gamma\subseteq\mathrm{SO(n)}$ at infinity is trivial. Otherwise, it could not act freely on $\R^n\setminus\left\{0\right\}$. However, this implies that $(M,g)$ is AE and contains a line, i.e. a geodesic that is minimizing between all of its points. In this situation, $\Ric=0$ and the Cheeger-Gromoll splitting theorem imply that it splits as $(M,g)=(\R\times N,dr^2+h)$ which we already excluded.
\item $(M,g)$ has only one end. Otherwise, it contains a line and Cheeger-Gromoll would again lead to a contradiction.
\item $\mathcal{H}_1(M)=\ker_{L^2}(\Delta_{H_1})=0$. This is due to the maximum principle, because $\Delta_{H_1}=\nabla^*\nabla$. 
\item $(M,g)$ admits at most finite quotients which are again ALE: If $(N,h)$ is an ALE manifold with $(M,g)$ as its universal cover, then we also have a covering map $\pi:M_{\infty}\to N_{\infty}$.
Because $M$ has only one end, $M_{\infty}$ is connected.
 Since $\pi_1(N_{\infty})$ is finite, $\pi$ is a finite cover which extends to a finite cover $\pi:M\to N$. Therefore $N=M/G$ with $G$ being a finite group.
\end{itemize}
\begin{remark}
 By restricting to $G$-invariant sections, all the heat kernel estimates we are going to establish in this section do also descend to $M/G$. Therefore, we may drop the assumption $\pi_1(M)=\left\{0\right\}$ and state the results as in Subsection \ref{subsubsec : geometric operators}.
\end{remark}
\begin{remark}
It is standard that any Ricci-flat manifold satisfying \eqref{holonomy} carries a parallel spinor. Morever, all the groups in \eqref{holonomy} actually appear as holonomy groups of Ricci-flat ALE manifolds, see \cites{Kro89,Joy99,Joy00,Joy01}.
On the other hand, there are no exmaples of Ricci-flat ALE manifolds which do not carry a parallel spinor.
It is a major open question whether there such examples exist, c.f.\ \cite{BKN89}*{p.\ 315}.
\end{remark}
\subsection{The twisted Dirac operator on spinor-valued one-forms}
Here, we are going to prove Corollary \ref{cor : heat flow ricci flat} (i) whose statement is contained in Theorem \ref{thm: heat flow twisted Dirac}.
We start this subsection with a short exposition on the twisted Dirac operator which is based on \cite{HS19}. We refer to this paper for further details.
We consider the bundle $ S\otimes T_{\C}^*M$ of spinor-valued one-forms.
	The spinor bundle naturally embeds in $S\otimes T_{\C}^*M$ by
	\[
	i:\sigma\mapsto -\frac{1}{n}\sum_{i=1}^n    e_i\cdot \sigma\otimes e_i^*
	\]
	and its image yields a subbundle, denoted by $S_{1/2}:=i(S)$. Let $\mu: C^{\infty}(M,S\otimes T_{\C}^*M)\to C^{\infty}(M,S)$ be defined by $ \sigma\otimes e_i^*\mapsto e_i\cdot \sigma$. 
	Then the bundle $S_{3/2}:=\ker(\mu)$ is the orthogonal complement of $S_{1/2}$ and we have the orthogonal projection maps 
	\begin{align*}
	\textrm{pr}_{S_{1/2}}&=i\circ\mu:C^{\infty}(M,S\otimes T_{\C}^*M)\to C^{\infty}(M,S_{1/2}),\\
	\textrm{pr}_{S_{3/2}}&=1-i\circ\mu:C^{\infty}(M,S\otimes T_{\C}^*M)\to C^{\infty}(M,S_{3/2}).
	\end{align*}
	\begin{definition}
Let $S\otimes T_{\C}^*M$ be equipped with the twisted Dirac operator $\D_{T^*M}$.
\begin{itemize}
\item[(i)] 	The operator
	\[ P:=\textrm{pr}_{S_{3/2}}\circ\nabla: C^{\infty}(M,S)\to C^{\infty}(M;S_{3/2})
	\] 
	is called Penrose operator or twistor operator.
	\item[(ii)] The operator
	\[ Q:=\textrm{pr}_{S_{3/2}}\circ \D_{T^*M}|_{S_{3/2}}: C^{\infty}(M,S_{3/2})\to C^{\infty}(M,S_{3/2})
	\] 	
	is called Rarita-Schwinger operator.
\end{itemize}
	\end{definition}	
	With respect to the decomposition $S\otimes T_{\C}^*M=S_{1/2}\oplus S_{3/2}$, the operator $\D_{T^*M}$ is written as
	\begin{align}\label{eq: Dirac block} \D_{T^*M}=\begin{pmatrix} \frac{2-n}{n}i\circ \D_S\circ i^{-1} & 2i\circ P^*\\ \frac{2}{n}P\circ i^{-1} & Q
	\end{pmatrix},
	\end{align}
	where $\D_S$ is the classical Dirac operator on spinors .
Therefore, $\D_{T^*M}^2$ decomposes as
	\[ (\D_{T^*M})^2=\begin{pmatrix} \frac{(2-n)^2}{n^2}i\circ \D_S^2\circ i^{-1}
	+\frac{4}{n}i\circ P^*\circ P\circ i^{-1} &
\frac{4-2n}{n^2}i\circ \D_S\circ P^*+2i\circ P^*\circ Q	
	\\ 
	\frac{4-2n}{n^2}P\circ \D_S\circ i^{-1}+\frac{2}{n}Q\circ P\circ i^{-1}
	 & Q^2+\frac{4}{n}P\circ P^*
	\end{pmatrix}.
	\]
In the Ricci-flat setting, we have
\begin{align}\label{eq : P and P^*}P^*\circ P=\frac{n-1}{n}\D_S^2, \qquad\qquad
\frac{2-n}{n}P\circ \D_S+Q\circ P=0.
\end{align}
 By introducing the standard Laplacians $\Delta_{S_{1/2}}\in \mathrm{Diff}_2(S_{1/2})$ and $\Delta_{S_{3/2}}\in \mathrm{Diff}_2(S_{3/2})$, we get the Weitzenb\"{o}ck formulas
\begin{align}\label{eq : P,Q and D 1}
\D_S^2=\Delta_{S_{1/2}},\qquad \qquad
Q^2+\frac{4}{n}P\circ P^*=0
\end{align}
and hence,
	\begin{align}\label{eq : twisted Dirac block} (\D_{T^*M})^2=\begin{pmatrix} i\circ \Delta_{S_{1/2}}\circ i^{-1}
	 &
0
	\\ 
	0 & \Delta_{S_{3/2}}
	\end{pmatrix}.
	\end{align}
\begin{prop}\label{prop: kernel twisted dirac}
We have  $\ker(\Delta_{S_{3/2}})\subset  O_{\infty}(r^{-n})$.
 Furthermore,
\[
\ker_{L^2} (\Delta_{S_{3/2}})= 
\ker_{L^2} (\D_{T^*M}^2).
\]
\end{prop}
\begin{proof}Recall that
\[
\mathrm{Hol}(M,g)\in\left\{\mathrm{SU}(n/2),\mathrm{Sp}(n/4),\mathrm{Spin}(7)\right\}.
\]
In all these cases, bundle isometries were constructed in \cite{HS19}. If $n=2m$ and $\mathrm{Hol}(M^n,g)=\mathrm{SU}(m)$ (i.e.\ $(M,g)$ is Calabi-Yau), we have (c.f.\ \cite{HS19}*{Sec.\ 4.6})
\[
S_{3/2}\cong \oplus_{k=0}^m(\Lambda^{1,k}T^*_{\C}M\oplus\Lambda^{n-k,1}T^*_{\C}M)\ominus \oplus_{k=0}^m\Lambda^{0,k}T^*_{\C}M.
\]
If $n=4m$ and $\mathrm{Hol}(M^n,g)=\mathrm{Sp}(m)$ (i.e.\ $(M,g)$ is hyperk\"{a}hler), 
we have (c.f.\ \cite{HS19}*{Sec.\ 4.7})
\[
S_{3/2}\cong \oplus_{k=0}^m2(n-k+1)(\Lambda^{k+1}_0T^*_{\C}M\oplus\Lambda^{k-1}_0T^*_{\C}M\oplus \Lambda^{k,1}_0T^*_{\C}M) \ominus \oplus_{k=0}^m(n-k+1)\Lambda^{k}_0T^*_{\C}M.
\]
Finally, if $n=8$ and $\mathrm{Hol}(M^n,g)=\mathrm{Spin}(7)$, we have (c.f.\ \cite{HS19}*{Sec.\ 4.9})
\[
S_{3/2}\cong T^*_{\C}M\oplus \Lambda^{3}_{4,8}T^*_{\C}M\oplus \Lambda^{4}_{3,5}T^*_{\C}M \oplus \Lambda^{2}_{2,1}T^*_{\C}M.
\]
In all of these cases, $\Delta_{S_{3/2}}$ coincides with $\Delta_H$ via the bundle isometry. 
Note that \cite{HS19} considers compact manifolds but these identifications also work in the non-compact setting as they are purely built on representation-theoretic arguments.
Because $\mathcal{H}_1(M)=\left\{0\right\}$, Proposition \ref{prop : decay harmonic forms} implies $\omega\in\O_{\infty}(r^{-n})$  for any $\omega\in C^{\infty}(M,\Lambda M)$ with $\omega\in L^p$ for some $p<\infty$. By splitting into real and imaginary part, the same assertion also holds for $\omega\in C^{\infty}(M,\Lambda_{\C} M)$.
Thus the first assertion of the proposition follows from the existence of these identifications of $\Delta_{S_{3/2}}$ and $\Delta_H$.
 We furthermore have
\[
\ker_{L^2} (\D_{T^*M}^2)=\ker_{L^2} (\Delta_{S_{1/2}})\oplus \ker_{L^2} (\Delta_{S_{3/2}})=\ker_{L^2} (\Delta_{S_{3/2}}).
\]
The first equality follows from the diagonal form of $\D_{T^*M}^2$.
Moreoever, \eqref{eq : Bochner Dirac} and \eqref{eq : P,Q and D 1} yield  $\Delta_{S_{1/2}}=(\D_S)^2=\nabla^*\nabla$ in this situation. Because $(M,g)$ has only one end, the maximum principle implies that $\ker_{L^2} (\Delta_{S_{1/2}})$ is trivial which proves the second equality.
\end{proof}
\begin{thm}\label{thm: heat flow twisted Dirac}On a Ricci-flat ALE spin manifold with a parallel spinor, the heat flows of $\D_{T^*M}^2$ and $\Delta_{S_{3/2}}$ satisfy heat kernel and derivative estimates of degree $0$.
\end{thm}
\begin{proof}By the proof of the previous proposition, we have $\Delta_{S_{3/2}}\cong \Delta_H$ via a parallel isomorphism of vector bundles. Furthermore, we have $\D_{T^*M}^2\cong(\Delta_{S_{1/2}},\Delta_{S_{3/2}})\cong ((\D_S)^2,\Delta_H)$. 
Thus, the assertion follows from Theorem \ref{thm : heat kernel Dirac} and Theorem \ref{thm : heat flow Hodge Laplace}.
\end{proof}
\subsection{The classical Dirac operator revisited}
In this subsection, we are going to prove the second assertion in Corollary \ref{cor : heat flow dirac operator} which is Theorem \ref{thm : heat kernel Dirac 2} below.
\begin{lem}\label{lem : commuting Dirac}
We have $\D_{T^*M}\circ\nabla =\nabla\circ \D_{S}$.
\end{lem}
\begin{proof}
Using a local orthonormal basis, we may write $\nabla\varphi=\sum_i\nabla_{e_i}\varphi\otimes e_i^*$ and calculate, using the curvature identity in \cite{Gin09}*{Lem.\ 1.2.4},
\begin{align*}
\D_{T^*M}\nabla\varphi=\sum_{i,j}e_j\cdot \nabla^2_{e_j,e_i}\varphi\otimes e_i^*&=\sum_{i,j}[\nabla_{e_i}(e_j\cdot\nabla_{e_i}\varphi)\otimes e_i^*+e_j\cdot R_{e_j,e_i}\varphi\otimes e_i^*]\\
&=\nabla (\D_S\varphi)+\sum_i\Ric(e_i)\cdot\varphi\otimes e_i^*.
\end{align*}
The result follows from $\Ric=0$.
\end{proof}
\begin{lem}\label{lem : estimating Dirac}
We have
\[
\nabla^*\nabla\leq (D_S)^2,\qquad \nabla\circ\nabla^*\leq \frac{n}{2}(D_{T^*M})^2.
\]
\end{lem}
\begin{proof}
The first identity is immediate from \eqref{eq : Bochner Dirac}. For proving the second inequality, we use the decomposition $S\otimes T_{\C}^*M=S_{1/2}\oplus S_{3/2}$ to write
\[
\nabla=(\mathrm{pr}_{S_{1/2}},\mathrm{pr}_{S_{3/2}})\circ\nabla= (i\circ \D_S,P).
\]
For $\psi\in C^{\infty}_{c}(M,S)$, written as $\psi=\psi_{1/2}+\psi_{3/2}$ with respect to this decomposition, we conclude
\[
\nabla^*\psi=\D_S\circ i^{-1}\psi_{1/2}+P^*\psi_{3/2}.
\]
An application of the triangle inequality and integration by parts yields
\begin{align*}
	(\nabla\circ\nabla^*\psi,\psi)_{L^2}
		&=\norm{\nabla^*\psi}_{L^2}^2 
			\leq 2\norm{\D_S\circ i^{-1}\psi_{1/2}}_{L^2}^2+2\norm{P^*\psi_{3/2}}_{L^2}^2\\
		&\leq 2\norm{\D_S\circ i^{-1}(\psi_{1/2})}_{L^2}^2+2\norm{P^*\psi_{3/2}}^2+\frac{n}{2}\norm{Q\psi_{3/2}}_{L^2}^2\\
		&\leq \frac{n}{2}(i\circ (\D_S)^2\circ i^{-1}(\psi_{1/2}),\psi_{1/2})_{L^2}+\frac n2 (Q^2(\psi_{3/2})+\frac{4}{n}P\circ P^*(\psi_{3/2}),\psi_{3/2})_{L^2}\\
		&=\frac{n}{2}((\D_{T*M})^2\psi,\psi)_{L^2},
\end{align*}
where we used \eqref{eq : P,Q and D 1}, and \eqref{eq : twisted Dirac block}.
\end{proof}
\begin{thm}\label{thm : heat kernel Dirac 2}
On a Ricci-flat ALE spin manifold with a parallel spinor, the operator $e^{-t(\D_S)^2}$ satisfies derivative estimates of degree $1$.
\end{thm}
\begin{proof}
This is a consequence of Lemma \ref{lem : commuting Dirac}, Lemma \ref{lem : estimating Dirac} and Theorem \ref{thm: heat flow twisted Dirac}.
\end{proof}

\subsection{The Lichnerowicz Laplacian}\label{subsec : LL}
Here, we are going to prove Corollary \ref{cor : heat flow ricci flat} (ii)  which is  the following:
\begin{thm}\label{thm: heat flow LL}
On a Ricci-flat ALE spin manifold with a parallel spinor, the heat flow of the Lichnerwicz Laplacian satisfies almost Euclidean heat kernel estimates and derivative estimates of degree $0$.
\end{thm}
\begin{proof}Let $\sigma\in C^{\infty}(S)$ be a parallel spinor normalized such that $|\sigma|=1$.
Consider the following endomorphism
\begin{align*}
\Phi:C^{\infty}(S^2M)\to C^{\infty}(S\otimes T_{\C}^*M),\qquad h\mapsto \sum_{i,j} h_{ij} e_i\cdot\sigma\otimes e_j^*.
\end{align*}
It is easy to see that
\begin{align*}
\langle \Phi(h),\Phi(k)\rangle=\langle h,k\rangle,\qquad |\nabla^k(\Phi(h))|=|\nabla^k h|,\qquad k\in\N.
\end{align*}
In other words, $\Phi$ is an isometric embedding and its image forms a parallel subbundle of $S\otimes T_{\C}^*M$.
The key formula was established by Wang \cite{Wang91} (and independently in \cite{DWW05}) and states that
\begin{align}\label{eq : LL commutes}
\Phi\circ \Delta_L=(\D_{T^*M})^2\circ \Phi.
\end{align}
Therefore
\begin{align}\label{eq: hk LL commutes}
\Phi\circ e^{-t\Delta_L}=e^{-t(\D_{T^*M})^2}\circ\Phi
\end{align}
and the assertion follows directly from Theorem \ref{thm: heat flow twisted Dirac}.
\end{proof}
The second theorem in this subsection concerns an improved decay for the linearized de-Turck vector field and the linearized Ricci curvature along the heat flow of the Lichnerowicz Laplacian.
Given two arbitrary metrics $\tilde{g},\hat{g}$ the de Turck vector field is 
\begin{align*}
V(\tilde{g},\hat{g})=g^{ij}(\tilde{\Gamma}_{ij}^k-\hat{\Gamma}_{ij}^k),
\end{align*}
and it is used to define the Ricci de Turck flow which is a strictly parabolic variant of the Ricci flow. Let 
\begin{align*}
\delta:C^{\infty}(M,S^2M)&\to C^{\infty}(M,T^*M),\qquad (\delta h)_k= -g^{ij}\nabla_ih_{jk},\\
\delta^*: C^{\infty}(M,T^*M)&\to C^{\infty}(M,S^2M),\qquad (\delta^*\omega)_{ij}=\frac{1}{2}(\nabla_i\omega_j+\nabla_j\omega_i)
\end{align*}
be the divergence and its formal adjoint, respectively.
Let furthermore
\begin{align*}
G:C^{\infty}(M,S^2M)&\to C^{\infty}(M,S^2M), \qquad G(h)=h-\frac{1}{2}\tr_gh\cdot g
\end{align*}
be the gravitational operator and $\sharp:C^{\infty}(M,T^*M)\to C^{\infty}(M,TM)$ be the musical isomorphism.
The linearization of $V$ in the first component can now be expressed in terms of these operators as
\begin{align*}
DV(h):=\frac{d}{dt}|_{t=0}V(g+th,g)=-\sharp\circ \delta\circ G(h).
\end{align*}
Moreover, one computes (using e.g.\ the formulas in \cite{Besse07}*{Thm.\ 1.174})
\begin{align*}
D\mathcal{L}_V(h)&:=\frac{d}{dt}|_{t=0}\mathcal{L}_{V(g+th,g)}(g+th)=-2\cdot\delta^*\circ\delta\circ G(h),\\
D\Ric(h)&:=\frac{d}{dt}|_{t=0}\Ric_{g+th}=\frac{1}{2}\Delta_Lh-\delta^*\circ\delta\circ G(h)=\frac{1}{2}(\Delta_L+D\mathcal{L}_V)(h)
\end{align*}
The key tool to get good estimates for these operators under the heat flow is the following lemma:
\begin{lemma}\label{lem: commutators}
Let $\Delta_{H_1}$ the Hodge Laplacian on one-forms, $\Delta_L$ the Lichnerowicz Laplacian, $\delta:C^{\infty}(S^2M)\to C^{\infty}(T^*M)$ and $\delta^*$ its adjoint. Then we have
\begin{align*}
	\Delta_L \circ \delta^*
		&=\delta^* \circ \Delta_{H_1},
		\qquad
	\delta \circ \Delta_L
		= \Delta_{H_1} \circ \delta, \\
	\delta\circ\delta^*
		&\leq \Delta_{H_1},
		\qquad
	\delta^*\circ\delta
		\leq \Delta_L.
\end{align*}
\end{lemma}
\begin{proof}
Recall that $0 = \Phi^*\Ric_{g} = \Ric_{\Phi^*g}$ for any diffeomorphism $\Phi$.
Linearizing this equation at $g$, and noticing that $\de^*\o = \frac12 \L_{\o^\sharp}g$, we get
\begin{align*}
	0 
		&= D\Ric(\de^*\o) = \frac{1}{2}\Delta_L (\de^*\o) - \de^*(\de + \frac{1}{2}\n\tr)\de^*\o = \frac{1}{2}\Delta_L (\de^*\o) - \frac{1}{2}\de^*\Delta_{H_1}\o,
\end{align*}
for all one-forms $\o$, where we have used that 
\[
	\Delta_{H_1} = \left(2\de +  \n \tr \right) \de^*,
\]
which uses $\Ric = 0$.
This proves the first assertion, the second assertion follows by the equality of the formal adjoints.
Using the last equation for $\Delta_{H_1}$, note that 
\begin{align*}
\delta\circ\delta^*\nabla f
	&=\frac{1}{2}\Delta_{H_1} \nabla f - \frac12 \n \tr\left(\de^* \n f\right) = \Delta_{H_1} \nabla f,\qquad f\in C^{\infty}(M),\\
\delta\circ\delta^*\omega
	&=\frac12\Delta_{H_1}\omega\in\ker(\delta),\qquad \omega\in C^{\infty}(T^*M)\cap\ker(\delta).
\end{align*}
Thus, $\delta\circ\delta^*$ preserves the $L^2$-orthogonal decomposition
\begin{align*}
	H^k(M,T^*M)
		=\overline{\nabla(H^{k+1}(M))}^{H^k}\oplus \ker_{H^k}(\delta),
\end{align*}
with respect to which we have $\delta\circ \delta^*=\mathrm{diag}(\Delta_{H_1},\frac12\Delta_{H_1})$. In particular, $\delta\circ \delta^*\leq \Delta_{H_1}$. For the second inequality, we use the formulas from above to compute
\begin{align*}
	\delta^* \circ \delta(\delta^*\nabla f)
		&=\frac12 \delta^*(\Delta_{H_1}\nabla f) = \frac12 \Delta_L(\delta^*\nabla f),\qquad f\in C^{\infty}(M),\\
	\delta^* \circ \delta(\delta^*\omega)
		&=\frac{1}{2}\delta^*\circ\Delta_{H_1}\omega=\frac{1}{2}\Delta_L(\delta^*\omega), \qquad \omega\in C^{\infty}(T^*M)\cap\ker(\delta).
\end{align*}
These equations show that $\Delta_L$ and $\delta^*\circ\delta$ both preserve the splitting
\begin{align*}
	H^k(M,S^2M)
		=\overline{\delta^*\nabla(H^{k+2}(M))}^{H^k}\oplus \overline{\delta^*(\ker_{H^{k+1}}(\delta))}^{H^k}\oplus \ker_{H^k}(\delta).
\end{align*}
The splitting is $L^2$-orthogonal for the following reason: The first two spaces are both orthogonal to the third one. 
Using integration by parts and the fact that $\delta\circ\delta^*$ preserves the above splitting of $H^k(T^*M)$, one easily sees that the first two factors are also orthogonal to each other.
Because we see that $\delta^*\delta=\mathrm{diag}(\frac12 \Delta_L,\frac12\Delta_L,0)$ with respect to this decomposition, we get the desired inequality.
\end{proof}
\begin{thm}On a Ricci-flat ALE spin manifold with a parallel spinor, we have, for all $p\in (1,\infty)$,
\begin{align*}
\norm{DV\circ e^{-t\Delta_L}}_{p \to p}\leq Ct^{-\frac{1}{2}},\qquad
\norm{D\mathcal{L}_V\circ e^{-t\Delta_L}}_{p \to p}\leq Ct^{-1},\qquad
\norm{D\Ric\circ e^{-t\Delta_L}}_{p \to p}\leq Ct^{-1}.
\end{align*}
\end{thm}
\begin{proof}
Because the endomorphism, $G\in C^{\infty}(M,\mathrm{End}(S^2M))$ commutes with $\Delta_L$, we have
\begin{align*}
DV\circ e^{-t\Delta_L}=\sharp\circ \delta\circ G\circ e^{-t\Delta_L}=\sharp\circ\delta\circ  e^{-t\Delta_L}\circ G.
\end{align*}
Because  of Lemma \ref{lem: commutators}, we can apply Theorem \ref{thm_special_derivative_estimates} to get\begin{align*}
\left\|DV\circ e^{-t\Delta_L}\right\|_{p \to p}\leq C\left\|\delta\circ  e^{-t\Delta_L}\right\|_{p \to p}\leq Ct^{-\frac{1}{2}},
\end{align*}
where we used that $\sharp$ and $G$ are both pointwise bounded, hence bounded on $L^p$.
For the proof of the second inequality, we write
\begin{align*}
D\mathcal{L}_V\circ e^{-t\Delta_L}=-2\delta^*\circ\delta\circ G\circ e^{-t\Delta_L}=-2\delta^*\circ e^{-\frac{t}{2}\Delta_{H_1}}\circ \delta\circ e^{-\frac{t}{2}\Delta_L}\circ G.
\end{align*}
Again due to Lemma \ref{lem: commutators}, Theorem \ref{thm_special_derivative_estimates} and boundedness of $G$, we get
\begin{align*}
\norm{D\mathcal{L}_V\circ e^{-t\Delta_L}}_{p \to p}\leq C\norm{\delta^*\circ e^{-\frac{t}{2}\Delta_{H_1}}}_{p \to p}\cdot \norm{\delta\circ e^{-\frac{t}{2}\Delta_L}}_{p \to p}
\leq C t^{-1}.
\end{align*}
For the proof of the last estimate,
we first use \eqref{eq : LL commutes} and \eqref{eq: hk LL commutes}
to get
\[
	\Phi\circ\Delta_L\circ e^{-t\Delta_L}
		=(\D_{T^*M})^2\circ e^{-t(\D_{T^*M})^2}\circ \Phi
		=\left(\D_{T^*M}\circ e^{-\frac{t}{2}(\D_{T^*M})^2}\right)^2\circ\Phi.
\]
Because the map $\Phi:C^{\infty}(M,S^2M)\to C^{\infty}(M,S\otimes T^*M)$ is an isometric embedding, Theorem \ref{thm_special_derivative_estimates} yields
\begin{align*}
\norm{\Delta_L\circ e^{-t\Delta_L}}_{p \to p}
\leq \norm{ \D_{T^*M}\circ e^{-\frac{t}{2}(\D_{T^*M})^2}}_{p \to p}^2\leq C t^{-1}.
\end{align*}
Because $2D\Ric=-\Delta_L+D\mathcal{L}_V$,
\begin{align*}
\norm{D\Ric\circ e^{-t\Delta_L}}_{p \to p}
\leq \frac{1}{2}\left(
\norm{D\mathcal{L}_V\circ e^{-t\Delta_L}}_{p \to p}+\norm{\Delta_L\circ e^{-t\Delta_L}}_{p \to p}\right)\leq Ct^{-1},
\end{align*}
which finishes the proof.
\end{proof}

\subsection{The Hodge Laplacian on one-forms revisited}\label{subsec : better 1-forms}
In this final subsection, we are proving Corollary \ref{cor : heat flow one forms} (iii) which is the following theorem: 
\begin{thm}\label{thm: heat flow HL}
On a Ricci-flat ALE spin manifold with a parallel spinor, the heat flow of the Hodge Laplacian on one-forms satisfies  derivative estimates of degree $1$.
\end{thm}
\begin{proof}
Via the canonical splitting $T^2M=S^2M \oplus  \Lambda^2M$ of $2$ tensors into the symmetric and antisymmetric part, the covariant derivative $\nabla: C^{\infty}(M,T^*M)\to C^{\infty}(M,T^2M)$ splits as
\begin{align*}
\nabla\omega\cong (\delta^*\omega,d\omega)
\end{align*}
and we have 
\begin{align*}\delta^*\circ\Delta_{H_1}=\Delta_L\circ\delta^*\qquad d\circ \Delta_{H_1}=\Delta_{H_2}\circ d
\end{align*}
and
\begin{align*}
\delta\circ\delta^*\leq \Delta_{H_1},\qquad \delta^*\circ\delta\leq \Delta_{L},\qquad d^*\circ d\leq \Delta_{H_1},\qquad d\circ d^*\leq \Delta_{H_2}.
\end{align*}
The result now follows from Proposition \ref{prop : better derivative estimates 0 introduction}, Theorem \ref{thm : heat flow Hodge Laplace} and Theorem \ref{thm: heat flow LL}.
\end{proof}

\end{sloppypar}
\end{document}